\documentclass[11pt]{amsart}
\usepackage{color, framed, amsfonts, amsmath, amscd,amssymb}
\usepackage{amsbsy, amsthm,  amstext, amsopn, verbatim,multicol, mathrsfs}

\usepackage{fancybox,calc,wasysym} 
\usepackage[mathscr]{eucal}
\usepackage[all]{xy}
 \usepackage{latexsym}
\usepackage{mathrsfs}
\usepackage{fullpage}
   \usepackage[mathscr]{eucal}
\definecolor{LightGray}{rgb}{0.96,0.97,0.98}
  \usepackage{hyperref}

\numberwithin{equation}{subsection}
\newtheorem*{theorem}{Theorem}
\newtheorem*{proposition}{Proposition}
\newtheorem*{lemma}{Lemma}
\newtheorem*{sublemma}{Subemma}
\newtheorem*{corollary}{Corollary}
\theoremstyle{definition}

\newtheorem*{definition}{Definition}

\newtheorem*{notation}{Notation}

\makeatletter
\renewcommand{\subsection}{\@startsection{subsection}{2}{0pt}{-3ex plus
-1ex minus -0.2ex}{-2mm plus -0pt minus -2pt}{\normalfont\bfseries}} \makeatother

\newcommand{\C}{\mathbb{C}}
\newcommand{\Z}{\mathbb{Z}}
\newcommand{\N}{\mathbb{N}}

\newcommand{\D}{{\mathcal D}}
\newcommand{\hreg}{{\mathfrak{h}}^{\operatorname{reg}}}

\DeclareMathOperator{\Ext}{\mathrm{Ext}}
\DeclareMathOperator{\SSExt}{\mathcal{EXT}}
\DeclareMathOperator{\SSHom}{\mathcal{HOM}}
\DeclareMathOperator{\SExt}{\mathcal{E}{\it xt}}
\DeclareMathOperator{\SHom}{\mathcal{H}{\it om}}
\DeclareMathOperator{\Hom}{\mathrm{Hom}}
\DeclareMathOperator{\End}{\mathrm{End}}
\DeclareMathOperator{\im}{\mathrm{Im}}
\DeclareMathOperator{\Ker}{\mathrm{Ker}}

\DeclareMathOperator{\gr}{\mathrm{gr}}

\DeclareMathOperator{\injdim}{\mathrm{injdim}}
\DeclareMathOperator{\chdim}{\mathrm{chdim}}

\DeclareMathOperator{\rGr}{{\ttt{Grmod}\text{-}\hskip -1.5pt}}
\DeclareMathOperator{\rqgr}{\ttt{qgr}\text{-}\hskip -2pt}

 \newcommand{\tR}{{\Delta}(R)}
\DeclareMathOperator{\Spec}{\ttt{Spec}}

 \newcommand{\ttt}{\textsf}
 
  \newcommand{\lgr}{\text{-}\ttt{grmod}}
 \newcommand{\lGr}{\text{-}\ttt{Grmod}}
 \newcommand{\lqgr}{\text{-}\ttt{qgr}}  
 \newcommand{\lQGr}{\text{-}\ttt{Qgr}}  
  \newcommand{\lMod}{\text{-}\ttt{Mod}}
  \newcommand{\lmod}{\text{-}\ttt{mod}}
  \newcommand{\ltors}{\text{-}\ttt{tors}}
    \newcommand{\lTors}{\text{-}\ttt{Tors}}

\newcommand{\h}{\mathfrak{h}}
\newcommand{\OO}{\mathcal{O}}
\newcommand{\Lambdaplus}{\Lambda}
\newcommand{\Lambdawide}{{\sf G}(\Lambda)}
\newcommand{\Gammaplus}{\Gamma}

\DeclareMathOperator{\hi}{Hilb^n\mathbb{C}^2}

\DeclareMathOperator{\Proj}{\mathrm{Proj}}
\DeclareMathOperator{\Char}{\mathbf{Char}}
\DeclareMathOperator{\Supp}{\mathrm{Supp}}

\DeclareMathOperator{\coh}{\mathrm{Coh}}
\DeclareMathOperator{\qcoh}{\mathrm{Qcoh}}

\DeclareMathOperator{\md}{-\mathrm{mod}}

\DeclareMathOperator{\codim}{\mathrm{codim}}
\DeclareMathOperator{\op}{\mathrm{tr}}
\DeclareMathOperator{\opp}{\mathrm{op}}

\newcommand{\hr}{\mathfrak{h}^{\rm{reg}}}
 \newcommand{\MM}{\mathcal{M}}
 \newcommand{\NN}{\mathcal{N}}
 \newcommand{\XX}{\mathcal{X}}
 \newcommand{\XXgamma}{\mathcal{X}_{\gamma}} 
  \newcommand{\XXmu}{\mathcal{X}_{\mu}} 
   \newcommand{\XXnu}{\mathcal{X}_{\nu}} 
 \newcommand{\KK}{\mathcal{K}}
 
  \newcommand{\RR}{\mathcal{R}}
   \newcommand{\VV}{\mathcal{V}}
 \newcommand{\EE}{\mathcal{E}}
\newcommand{\ee}{{\epsilon_{jr}}}
\newcommand{\evee}{{\epsilon_{jr}^\vee}}
  \newcommand{\ve}{{\nabla}} 
  \newcommand{\kk}{\Bbbk}

\newenvironment{boxcode}{\VerbatimEnvironment%
\hskip 8pt
  \begin{Sbox} 
  \begin{minipage}{\linewidth-12\fboxsep-2\fboxrule-4pt}    
}{
  \end{minipage}   
  \end{Sbox} 
  \fcolorbox{black}{LightGray}{\TheSbox} 
} 
 
 \makeindex 
\begin{document} \title{The Auslander-Gorenstein property for $\Z$-algebras}
  \author{I. G. Gordon} \address{(IGG) School of Mathematics and Maxwell Institute for Mathematical Sciences,
Edinburgh  University, Edinburgh EH9 3JZ, Scotland.}
\email{igordon@ed.ac.uk}
\author{J. T. Stafford}
\address{(JTS) School of Mathematics, Alan Turing Building, The University of Manchester, Oxford Road, Manchester M13 9PL,
England.}
\email{Toby.Stafford@manchester.ac.uk}

\subjclass{Primary: 14A22,   16A62, 16W70, 18G40, 20C08}  
   \keywords{Auslander-Gorenstein algebra, double Ext spectral sequence, Cherednik
algebra,   equidimensionality of characteristic varieties. }
   \thanks{The first   author is grateful for the financial support of EPSRC grant EP/G007632. The second
   author was a Royal Society Wolfson Research Merit Award Holder while this research was conducted.}    

  \begin{abstract}     
We provide a framework for  part of the homological theory of $\mathbb{Z}$-algebras and their generalizations, directed towards analogues of the Auslander-Gorenstein condition and the associated double Ext spectral sequence that are useful for enveloping algebras of Lie algebras and related rings. As an application, we prove the  equidimensionality of the characteristic variety of an irreducible representation of the $\mathbb{Z}$-algebra, and for related representations over quantum symplectic resolutions.  
 In  the special case of Cherednik algebras of type A, this answers a question raised by the authors.
  \end{abstract}
   \maketitle
\tableofcontents

 \section{Introduction}

  \subsection{}\label{intro-1}  Throughout the paper, all rings will be algebras over a fixed base field $\kk$.
  An unadorned tensor product $\otimes$ will denote a tensor product over $\kk$.
  A \emph{lower triangular $\Z$-algebra} is a $\Z$-bigraded associative  algebra $$S = \bigoplus_{0\leq q\leq p\in \Z} S_{p,q}$$ with
 matrix-style multiplication and where each non-zero $S_{q,q}$ is a noetherian, unital   algebra
 and the $S_{p,q}$ are finitely generated  modules over both $S_{p,p}$ and $S_{q,q}$. 
 
 These algebras are important in noncommutative algebraic geometry -- see, for example, \cite{SV} or \cite{VdB}. They are also useful    in geometric representation theory, including in the study of   rational Cherednik algebras \cite{GS1,GS2},  deformed preprojective algebras \cite{Boy,Mus},   finite $W$-algebras \cite{Gi} and  deformations of  conical  symplectic singularities \cite{BPW}.
 In applications, the  $\mathbb{Z}$-algebra provides an effective way to relate the representation theory of the given algebra  to the geometry of  a resolution of singularities for its associated graded ring: indeed, 
  the $\mathbb{Z}$-algebra can be regarded 
  as a quantization of that resolution. In the above examples,  the $\mathbb{Z}$-algebras    quantize 
 Hilbert schemes of points on the plane,  minimal resolutions of Kleinian singularities, resolutions of Slodowy slices and, most generally, symplectic resolutions of conical symplectic singularities, respectively.

   What is missing, and what is provided in this paper, is a suitable homological machine to relate the commutative and noncommutative  theories. Our aim is  show how the Auslander-Gorenstein condition and the related double Ext spectral sequence, that are so useful 
 for   enveloping algebras and related rings,  can be generalised to work for $\mathbb{Z}$-algebras.
 As an application we generalize Gabber's equi\-dimensionality result to $\mathbb{Z}$-algebras.
 
  \subsection{}\label{intro-2}   
  To explain our results in more detail, we need some notation. 
  Write  $S_{0,0}\lmod$ for  the category of noetherian left $S_{0,0}$-modules, 
  $S\lgr$ for the category of noetherian left $S$-modules and    $S\lqgr$ for  the quotient category  of $S\lgr$ modulo the 
 bounded modules. Let $\pi_S:S\lgr\to S\lqgr$ be the natural projection.
    If the   $S_{p,q}$   are $(S_{p,p},\,S_{q,q})$-progenerators, we say that $S$ is a \emph{Morita $\mathbb{Z}$-algebra}. In this case,  for any $q\geq 0$,  there is
  an equivalence of categories $ S_{q,q}\lmod \buildrel{\sim}\over{\longrightarrow}  S\lqgr$ given by
  $\Phi: M\mapsto \pi_S( S_{\ast,q} \otimes_{S_{q,q}} M)$,   where
    $S_{\ast,q}=\bigoplus_{p\geq q} S_{p,q}$.
  
   We assume that  there is a  filtration $F$ on $S$ such that 
 the associated graded algebra $\gr_FS=\bigoplus \gr_FS_{p.q}=\tR$. Here,
  $\tR=\bigoplus \tR_{p,q}$ is  the $\Z$-algebra associated to some finitely generated commutative graded algebra
   $R=\bigoplus R_n$ by defining  $\tR_{p,q}=R_{p-q}$ for all $p\geq q$.  There are natural equivalences $\tR\lqgr \simeq R\lqgr  \simeq \coh(X)$, for $X=\Proj(R)$.
If  $M\in S_{00}\lmod$ has a good filtration $F$,  then $\mathcal{M}=\Phi(M)$ 
 is naturally filtered 
and has  associated graded sheaf $\gr_F\mathcal{M}\in \coh(X)$. The \emph{characteristic variety} $\Char(M)\subseteq X$
 of $M$ (or of  $\mathcal{M}$) is then the support of $\gr_F\mathcal{M}$.
 
  \subsection{}\label{intro-25}   If, for example, $S_{0,0}=U_\lambda$ is the spherical subalgebra of a rational Cherednik algebra of type $A$ -- see  Subsection~\ref{97} for the definitions -- then $S$ is defined by setting $S_{p,p}=U_{\lambda+p}$, with 
 a canonical choice of $(U_{\lambda+p},U_{\lambda+q})$-bimodules  $S_{p,q}$. In this case $S$ is  a Morita $\mathbb{Z}$-algebra whenever $\mathcal{R}e(\lambda)\geq\frac{1}{2}$. Moreover,     $X=\sf{Hilb}^n(\mathbb{C}^2)$ is the Hilbert scheme of $n$ points on the place and the natural map $X\to Y= \Spec(R_0)$ is a resolution of singularities of the quotient variety $(\mathbb{C}^{n}\times \mathbb{C}^{n})/S_n$, see Section~\ref{symplectic-sect} or \cite{GS1} for more details. The structure of $\Char(M)$ is complex: when $M$ is an irreducible 
 $U_\lambda$-representation from category $\mathcal{O}$ it has irreducible components parametrized by partitions of $n$, see \cite[Section~6]{GS2}.

  \subsection{}\label{intro-3}  
  This motivates the  question, raised in \cite{GS2} and again in \cite{Ro} and answered here, of whether $\Char(M)$ is equidimensional for simple $S_{0,0}$-modules $M$.  To answer this question we follow Gabber's proof for unital algebras, as described in \cite{Le}. This uses the 
  notion of a filtered  Auslander-Gorenstein algebra
   $U$ and the
  convergent  spectral sequence  
  $$\Ext^p_U(\Ext^q_U(M,U),U)\Rightarrow \mathbb{H}^{p-q}(M) = \begin{cases} M&\text{if} \ p=q\\
  0 & \text{otherwise.}\end{cases}$$
   
 In order to generalise this spectral sequence to $\mathbb{Z}$-algebras, we need extra conditions.
Suppose now that $S = \bigoplus_{p\geq q\geq 0} S_{p,q}$ and $T =  \bigoplus_{p\geq q\geq 0} T_{p,q}$ are two 
 filtered $\Z$-algebras with $S_{0,0} = T_{0,0}^{op}$.  We say that $q_S$ is a \emph{good parameter for
  $S$} if $S_{\geq \lambda }=\bigoplus_{  p\geq q\geq q_S}S_{p,q}$ is a Morita $\mathbb{Z}$-algebra. 
  We consider the following hypotheses.

\begin{itemize}
\item[(H1)] There exists a good parameter $q_S$ for $S$.
\item[(H2)] $\gr S \cong \tR$ where $R = \bigoplus_{n\geq 0} R_n$ is a finitely generated commutative   domain satisfying 
$R_m R_n = R_{m+n}$ for all $m,n \geq 0$.
\item[(H3)] Conditions (H1) and (H2) hold for $T$ with $q_S = q_T$. Furthermore, under the tensor product filtration defined in Notation~\ref{filt-notation}, 
$\gr (S_{p,0}\otimes_{S_{0,0}} T_{q,0}) \cong R_{p+q}$ for all $p,q \geq q_S$.
\item[(H4)] $X = \Proj(R)$ is Gorenstein, $\Spec (R_0)$ is normal, and the canonical morphism $X\rightarrow \Spec (R_0)$ is birational.
\end{itemize}

 \subsection{}\label{intro-4}  As we will describe in \ref{examples-intro} there are a number of important examples of 
 $\mathbb{Z}$-algebras arising in 
geometric representation theory that satisfy these hypotheses.  For these algebras the main results of the paper culminate in the following theorem, which summarises  Theorems~\ref{Bj4.15-NC} and \ref{main-thm}.

\begin{theorem}
Suppose that $S$ and $T$ are $\Z$-algebras that satisfy   Hypotheses (H1--H4). 
\begin{enumerate}
\item There exists a  module $\mathcal{X}\in (S\otimes T)\lqgr$ such that, for $\mathcal{M}\in S\lqgr$,  there is a convergent spectral sequence
$$\SSExt^p_T(\SSExt^q_S(M,\mathcal{X}),\mathcal{X})\Rightarrow \mathbb{H}^{p-q}(\mathcal{M}) 
= \begin{cases} \mathcal{M}&\text{if} \ p=q\\
  0 & \text{otherwise.}\end{cases}$$
\item If $\mathcal{M}\in S\lqgr$ and $\mathcal{N}$ is a $T$-submodule of 
$ \SSExt^q_S(M,\mathcal{X})$,
then $\SSExt^p_T(M,\mathcal{X})=0$ for $p<q$.
\item If   $ \Phi(M) \in S\lQGr$ is irreducible for some $M\in S_{q,q}\lmod$, then $\Char(M)$ is equidimensional.
\end{enumerate}
\end{theorem}

  \subsection{}\label{intro-5}  
 Let us explain some of the undefined terms in this theorem, and the significance of the hypotheses (H1--H4). First, 
 $\SSExt$
 is the   analogue for   $S\lqgr$ of sheaf Ext   and is defined in Definition~\ref{ext-defn}.
 Part (2) of the theorem is then the natural analogue  of the Auslander-Gorenstein condition for $\mathbb{Z}$-algebras.  
 Given the earlier discussions, conditions (H1), (H2) and (H4) are natural and so it is only the requirement of an auxiliary 
 algebra $T$ and the description of $\mathcal{X}$ that require explanation.
 
 The problem comes  in generalizing $\Ext^p(M,U)$. One choice is to consider 
 $\mathcal{A}=\SSExt^p(\mathcal{M}, S)=\bigoplus_n \Ext^p_S(\mathcal{M},S_{\ast,n})$. There are two problems here. First,   one really wants to take the image of this module in a quotient category ${qgr}$. However, finitely generated right $S$-modules are bounded and so 
 $\rqgr S=0$ (see \ref{lackofsymmetry2})! Thus one does need an auxiliary algebra $T$. For example,  $\mathcal{A}$ is   naturally a  right module over the \emph{upper triangular} $\mathbb{Z}$-algebra $S^*=\bigoplus_{0\leq q\leq p}S_{p,q}^*$, where  $S_{p,q}^*=\Hom_{S_{q,q}}(S_{p,q},\, S_{q,q})$.   
 For more technical reasons  $T=S^*$ does not work; basically because  
 one would then need  $\gr(S_{p,q}^*)\otimes \gr(S_{p,q})=\gr(S_{q,q})$ for $p\geq q$, and this essentially never holds. 
 Fortunately, 
 in all the standard examples there is a natural candidate for $T$. For example given  
 the  Cherednik algebra  as  in \ref{intro-25}, where  $S$ is defined in terms of the sequence of algebras
  $S_{p,p}=U_{\lambda+p}$,  the algebra $T$ is defined by moving in the ``opposite'' direction: $T_{p,p}=U_{\lambda-p}^{op}$.  
  The module $\mathcal{X}$ is now the image in $ (S\otimes T)\lqgr$ of the 
  natural $S\otimes T$-bimodule $S_{\ast,0}\otimes_{S_{0,0}}T_{\ast,0}$.

  \subsection{}\label{examples-intro}
The hypotheses (H1--H4) hold for the example of rational Cherednik algebras and Hilbert schemes, thanks 
mostly to \cite{GGS}. In Proposition~\ref{smpl-hold},  we show that they also hold for the  $\Z$-algebras constructed in \cite{BPW}. 
These algebras are attached to $\C^*$-equivariant symplectic resolutions of singularities $X\rightarrow Y$, where $Y$ is an irreducible affine symplectic singularity with a $\C^*$-action that both attracts to a unique fixed point and scales the given symplectic form $\omega$ by $t\cdot \omega = t^m \omega$ for all $t\in \C^*$ and some $m>0$. Examples include:
\begin{itemize}
\item The minimal resolution  $\mathfrak{M}\to \mathfrak{M}_0=\mathbb{C}^2/\Gamma$  for a finite subgroup
 $\Gamma\subseteq SL(2,\mathbb{C})$ of type $A$;
\item  $X= {\sf Hilb}^n(\mathfrak{M})$ and $ Y ={\sf Sym}^n(\mathfrak{M}_0)$ with $\mathfrak{M}$ and $\mathfrak{M}_0$ as above;
\item The Springer resolution $T^*(G/B) \rightarrow \mathcal{N}$. More generally, one takes  $X=T^*(G/P)$, where $P$ is a parabolic subgroup of the reductive algebraic group $G$ and $Y$ is the
 affinization  of $X$;
 \item Various Nakajima quiver varieties and their affinizations.
\end{itemize}
In these examples, the $\Z$-algebra is constructed from a $\C^*$-equivariant deformation quantization of $X$ depending on $\lambda\in H^2(X,\C)$ and a very ample line bundle $\mathcal{L}$ on $X$. The characteristic variety can be interpreted, in this case, in terms of the deformation quantization.

  \subsection{}\label{intro-55}  
  It may be worth noting that we  do not   require   $X\rightarrow \Spec(R_0)$ in (H4) to be  a resolution of singularities. So it  makes sense to try to apply  Theorem~\ref{intro-4} in the case which arises from the minimal model programme, namely when the morphism $X\rightarrow Y$ is crepant with $X$ $\mathbb{Q}$-factorial and terminal.

\subsection{Outline of the paper.}\label{outline} The natural generality for the results of this paper are for algebras indexed by $\mathbb{Z}^n\times\mathbb{Z}^n$ for $n\geq 1$. Following \cite{Gi} we call them  \emph{directed algebras}; this generality has appeared for   \cite{Gi} and \cite{Mus} and has potential applications for other, more general symplectic reflection algebras. 
The essentials of directed algebras,
  combined with basic results about their Ext groups, filtrations and associated graded modules are given in Sections~\ref{section-directed}, \ref{homtech} 
  and \ref{filtration-section}. The analogue of dualizing and the spectral sequence relating the 
  cohomology of a  module to that of its associated graded module appears in Section~\ref{section-dualising}.
  This is used in Section~\ref{section-spectral} to prove parts (1) and (2) of Theorem~\ref{intro-4}.
  The  analogue of Gabber's  theorem,  Theorem~\ref{intro-4}(3), is then proved for general directed algebras in Section~\ref{section-main}. 
The application of the theorem to a number of different quantizations of symplectic singularities, including Cherednik algebras,
 is given in Section~\ref{symplectic-sect}.  
 
\clearpage 

\section{Directed algebras.}\label{section-directed}

\subsection{} \label{2.1}  As  remarked in the introduction,  the results in  this paper work for more than just 
 $\mathbb{Z}$-algebras,  and so  in this section we provide the relevant definitions and basic results.

Let  
  $\Lambdaplus $  be a submonoid of  a finitely generated free additive abelian group $\Lambdawide$ such that $\Lambdaplus$
   generates $\Lambdawide$ and
$\Lambdaplus\cap \, - \Lambdaplus = \{0\}$.  For $\lambda , \mu \in \Lambdawide$ 
we write $\lambda \geq \mu$ if $\lambda - \mu \in \Lambdaplus$. Note that if 
$\mu_1, \ldots , \mu_n \in \Lambdawide$ 
then there exists $\lambda\in \Lambdaplus$ such that $\lambda \geq \mu_i$ for each $i$, namely $\lambda = \sum \mu_i$.

\begin{definition}  
A {\it $\Lambdaplus$-directed algebra} \index{(Upper, lower) $\Lambda$-directed algebra} is a $\Lambdaplus$-bigraded $\kk$-algebra  
$$S = \bigoplus_{\lambda, \mu \in \Lambdaplus} S_{\lambda , \mu}$$ such that  each $S_{\lambda, \lambda}$ is a $\kk$-algebra 
and multiplication is defined matrix style:  $S_{\lambda, \mu} S_{\mu , \tau} \subseteq S_{\lambda , \tau}$ and  
$S_{\lambda, \mu} S_{\nu , \tau} =0$ for $\mu\not=\nu$.
We further assume that each $S_{\lambda,\lambda}$ is unital, with unit  $1_{\lambda}$,
  and that each $S_{\lambda,\mu} $ is finitely generated and unital as a left module over $S_{\lambda,\lambda}$ and as a right module over $S_{\mu,\mu}$.

 The  $\Lambdaplus$-directed algebra $S$  is  called a  {\it lower} $\Lambdaplus$-directed algebra if $S_{\lambda , \mu} \neq 0$ 
only if $\lambda \geq \mu$, respectively an   {\it upper} $\Lambdaplus$-directed algebra if $S_{\lambda , \mu} \neq 0$ 
only if $\lambda \leq \mu$. \end{definition}

\subsection{} {\bf Examples.}\label{Ex2.2}
There are three examples of  $\Lambdaplus$-directed algebras that will interest us.

\begin{enumerate} 
\item Let $R = \bigoplus_{\lambda\in\Lambdaplus} R_{\lambda}$ be a $\Lambdaplus$-graded $\kk$-algebra. Set $R_\lambda = 0$ for  
$\lambda\in \Lambdawide\smallsetminus \Lambdaplus$. Define  a lower $\Lambdaplus$-directed algebra $\tR$ by \index{$\tR$, the directed algebra of a $\Gamma$-graded algebra}
$\tR = \bigoplus_{\lambda,\mu\in \Lambda} \tR_{\lambda, \mu}$ where
$\tR_{\lambda, \mu} = R_{\lambda - \mu}$.  
 
\item  If $S$ is an upper $\Lambdaplus$-directed algebra, then its ``transpose''  $S^{\op}$ is a lower $\Lambdaplus$-directed algebra. Formally  
$S^{\op} = \bigoplus (S^{\op})_{\lambda, \mu}$ where    $(S^{\op})_{\lambda, \mu} = S_{\mu, \lambda}$, with the opposite multiplication.
 \item If  $S$ is lower $\Lambdaplus$-directed   and $T$ is lower $\Gammaplus$-directed, then $S\otimes T$ is 
lower $(\Lambdaplus\times \Gammaplus)$-directed. 
\end{enumerate}

In this paper we will only consider upper and lower  directed algebras. By (2)   it 
suffices to consider only the latter case.  
\medskip

\begin{boxcode}For the  rest of this section we assume that   
$S$ is a  lower $\Lambdaplus$-directed algebra.\end{boxcode}

\subsection{Definition} \label{graded-defn}
A \emph{graded (left) $S$-module}\index{$S\lGr$, the category of graded $S$-modules}
is a left $S$-module  ${\sf{M}} = \bigoplus_{\lambda \in \Lambdaplus} {\sf M}_{\lambda}$ with
 matrix style multiplication $S_{\lambda,\mu}{\sf M}_\nu=0$ for $\nu\not=\mu$ and  $S_{\lambda,\mu}{\sf M}_\mu\subseteq {\sf M}_\lambda$.
 Each ${\sf M}_{\lambda}$ is assumed to be a unital left $S_{\lambda,\lambda}$-module. 
 
 The 
 category of all such modules will be denoted $S\lGr$,  where the morphisms are the homogeneous $S$-homomorphisms. 
 The category of graded right modules is denoted $\rGr S$.
\smallskip

 If $S=\tR$ as in Example~\ref{Ex2.2}(1), then $\tR\lGr$ is the category of $\Lambdaplus$-graded $R$-modules. 
\subsection{}\label{hommumjum}    The category $S\lGr$ admits direct limits and these direct limits preserve exactness. 
Moreover the category has a set of distinguished objects parametrised by $\Lambdaplus$,
\begin{equation} \label{projgen}
 {\sf S}_{\ast, \lambda}\  := \  \bigoplus_{\tau\in\Lambdaplus } S_{\tau, \lambda}
\ = \ S\cdot 1_\lambda.
\end{equation}
\index{ ${\sf S}_{\ast, \lambda}$, distinguished objects in $S\lGr$} The set   $\{{\sf S}_{\ast,\lambda} : \lambda \in \Lambdaplus\}$   generates the category $S\lGr$: if  
${\sf{M}} = \bigoplus_{\lambda \in \Lambdaplus} {\sf M}_{\lambda}   \in S\lGr$ and $m\in {\sf M}_\lambda$, then
 $1_\lambda\mapsto m$ induces a unique homomorphism ${\sf S}_{\ast,\lambda}\to {\sf M}$.
 It follows from \cite[Chapitre~II, \S 6, Th\'eor\`eme~2]{Ga} that every object of $S\lGr$ has an injective hull.

 \begin{definition}\label{projgen1}\index{$S\lgr$, the category of noetherian graded $S$-modules}
 $S$ is  called {\it locally left noetherian} if ${\sf S}_{\ast,\lambda}$ is noetherian in $S\lGr$ for all $\lambda\in \Lambdaplus$.  The full subcategory of noetherian objects in $S\lGr$ will be denoted by $S\lgr$.
  \end{definition} 
  
  If $R$ is a    $\Lambdaplus$-graded algebra, then $\tR_{\ast,\lambda}  = R[-\lambda]$ where $R[-\lambda]_{\tau}= R_{\tau - \lambda}$ for all $\tau$. So $\tR$ is locally left noetherian   if and only if $R$ is  left noetherian. In general, when $S$
   is locally left noetherian, $S\lgr$ is the  subcategory of finitely generated modules.

\subsection{}\label{lackofsymmetry}  As    for unital graded algebras, it is natural to consider the quotient category 
  of graded noetherian modules modulo the Serre  subcategory of torsion modules.

\begin{definition} An object ${\sf M} \in S\lgr$ is called {\it torsion} if there exists $\lambda \in \Lambdaplus$ such that  ${\sf M}_{\tau} = 0$  
for all $\tau \in \lambda+ \Lambdaplus$. An object ${\sf M}$ of $S\lGr$ is {\it torsion} if it is
 the direct limit of noetherian torsion objects. We denote the corresponding full subcategories by $S\ltors$ and $S\lTors$ 
 respectively.
\end{definition}

The category $S\ltors$ is a Serre subcategory of $S\lgr$ and so we can form the quotient category 
$S\lqgr = S\lgr/S\ltors$. If $S$ is
 locally left noetherian (which is the main case of interest in this paper) then $S\lTors$ is  a localising subcategory of $S\lGr$,
  \cite[Chapitre~III, \S 3, Corollaire~1]{Ga} and so we
  can form the quotient category $S\lQGr = S\lGr/S\lTors$. By  \cite[Chapitre~III, \S1\&2]{Ga},
   $S\lQGr$ \index{$S\lQGr$, the category of quasicoherent sheaves attached to $S$}has an  exact quotient functor 
  $\pi_S: S\lGr \to S\lQGr$ whose right adjoint is  the section functor  $\sigma_S: S\lQGr \to S\lGr$.
 
\subsection{Vorsicht!}\label{lackofsymmetry2}
There is a substantial lack of symmetry in   concepts from the last two subsections. For example,
if $T$ is lower $\N$-directed  then every  graded  noetherian \emph{right} $T$-module is  torsion.

\subsection{Shifting} \label{shiftlemma} Let $S$ be lower $\Lambdaplus$-directed. For $\lambda\in \Lambdaplus$  
we define a new lower $\Lambdaplus$-directed algebra by
 $$S_{\geq \lambda} := \bigoplus_{\mu, \tau\in\Lambda} S_{\mu+\lambda, \tau+\lambda}.$$ We 
 have shift functors
  $$[\lambda] : S\lGr \longrightarrow S_{\geq \lambda} \lGr , \qquad [-\lambda] : S_{\geq \lambda}\lGr \longrightarrow S \lGr$$ 
  defined as follows. Given ${\sf M}\in S\lGr$, set 
$ {\sf M}[\lambda] = \bigoplus_{\tau\in \Lambdaplus}  {\sf M}[\lambda]_{\tau}$ where ${\sf M}[\lambda]_{\tau} = {\sf M}_{\lambda + \tau}$,
 while if  ${\sf N}\in S_{\geq \lambda}\lGr$ set ${\sf N}[-\lambda] = \bigoplus_{\tau\in \Lambdaplus}  {\sf N}[-\lambda]_{\tau}$ where 
${\sf N}[-\lambda]_{\tau} = {\sf N}_{\tau - \lambda}$  if $\tau\geq \lambda$ and is zero otherwise. It is immediate that 
${\sf N}[-\lambda][\lambda] = {\sf N}$, whilst there is a short exact sequence 
$$0 \longrightarrow {\sf M}[\lambda][-\lambda] \longrightarrow {\sf M} \longrightarrow {\sf C} \longrightarrow 0$$
 where   ${\sf C}_{\tau} = 0$ for all $\tau\geq \lambda$  and so ${\sf C}\in  S\lTors$.  This has the following useful
  consequences.

\begin{lemma}  Let $S$ be a locally left noetherian lower $\Lambdaplus$-directed algebra. 
\begin{enumerate}
  \item[(i)] The functors $[\lambda]$ and $[-\lambda]$ induce equivalences of categories 
between  $S \lQGr$ and
 $S_{\geq\lambda}\lQGr$  and hence between  the noetherian subcategories $S\lqgr$ and $S_{\geq\lambda}\lqgr$.
 
\item[(ii)] Suppose that there exists $\lambda \in \Lambdaplus$ such that $S_{\geq \lambda} = \tR$ for some  $\Lambdaplus$-graded commutative algebra $R$. Then  $S\lQGr$ and $R\lQGr$ are equivalent categories. 
 \qed
 \end{enumerate}
\end{lemma}
 
\subsection{Morita directed algebras}
Let $S$ be a lower $\Lambdaplus$-directed algebra. 

\begin{definition}\label{good-defn}
We call $\lambda\in \Lambdaplus$ a {\it good parameter} \index{Good parameter for a directed algebra}for $S$ and 
 $S_{\geq \lambda}$ a {\it Morita $\Lambdaplus$-directed algebra}
provided:
\begin{enumerate}
\item $S_{\lambda,\lambda}$ is a noetherian ring; 
\item for all $\mu\geq \lambda$,    tensoring with the 
$(S_{\mu, \mu}, S_{\lambda, \lambda})$-bimodule $S_{\mu,\lambda}$ induces a Morita equivalence 
 between $S_{\lambda, \lambda}\lMod$ and
   $S_{\mu, \mu}\lMod$;
    \item  for any
  $\mu\geq \tau \geq \lambda$
  the multiplication map $S_{\mu, \tau}\otimes_{S_{\tau,\tau}} S_{\tau, \lambda} \longrightarrow S_{\mu, \lambda}$ is an isomorphism. 
  \end{enumerate}
 \end{definition} 
    A routine application of Morita theory shows that if $\lambda$ is good then 
    so is any $\mu \geq \lambda$. 
If $\lambda $ is a good parameter and  $\mu\geq \lambda$ we write
 $S_{\mu,\lambda}^* = \Hom_{S_{\lambda,\lambda}}(S_{\mu,\lambda},\,S_{\lambda,\lambda})$ for the dual of $S_{\mu,\lambda}$. 
 By hypothesis it is an $( S_{\lambda,\lambda},\, S_{\mu,\mu})$-bimodule that is projective on both sides.

\subsection{}  A mild generalisation of  \cite[Lemma~5.5]{GS2} or \cite[Theorem~12]{Boy} gives:

\begin{proposition}\label{corner-ring} 
Suppose that  $S$ is a locally left noetherian lower $\Lambdaplus$-directed algebra and suppose that 
$\lambda$  is good. Then there exists an equivalence of categories $\Psi: S_{\lambda, \lambda}\lMod 
   \buildrel{\sim}\over\longrightarrow  S\lQGr$    given by
   ${\sf M}\mapsto \pi_S\left(\big(\bigoplus_{\tau} S_{\tau+\lambda, \lambda}\otimes_{S_{\lambda,\lambda}} {\sf M}\big)[-\lambda]\right).$ 
 It restricts to an equivalence between  $S_{\lambda, \lambda}\lmod$ and $S\lqgr$.
\end{proposition}

\begin{proof}  
We first claim that  the functor 
$\Psi':{\sf M}\mapsto \bigoplus_{\tau}  S_{\tau+\lambda, \lambda}\otimes_{S_{\lambda,\lambda}} {\sf M}$ provides  an
 equivalence $ S_{\lambda, \lambda}\lMod  \buildrel{\sim}\over{\longrightarrow}\ S_{\geq \lambda  } \lQGr $.
   This  follows from 
   the  cited references, after the following minor modifications.  First, they only deal with the
    $\mathbb{N}$-directed case, but the proof extends trivially. Second, those papers only prove the 
    result at the level of noetherian objects. This  implies the general case  since the functor is
     defined on all modules and, by the Morita equivalence,  it commutes with direct limits.

 The equivalence $ S_{\geq \lambda  } \lQGr  \buildrel{\sim}\over{\longrightarrow}\  S\lQGr$ and hence the proposition now follow from 
  Lemma~\ref{shiftlemma}.
   \end{proof}
 
 \subsection{Remark}\label{2.8.1} The functor $\Phi$ 
 inverse to $\Psi$, is easy to describe for  noetherian objects. 
  Specifically, if $\mathcal{M}=\pi_S\bigl(\bigoplus_{\mu\in \Lambda} {\sf M}_\mu\bigr) \in S\lqgr$
 then $\Phi(\mathcal{M}) = S^*_{\mu,\lambda}\otimes {\sf M}_\mu\in S_{\lambda,\lambda}\lmod$ for any $\mu\gg \lambda$. 
The details can be found in   \cite[Lemma~5.5]{GS2}.

\section{Homological notions.} 
\label{homtech}
  
  \medskip
\begin{boxcode}
In this  section we assume  that $S$,   $T$ and  $S\otimes T$ are, respectively,   $\Lambda$-directed,
  $\Gamma$-directed and  $(\Lambda \times \Gamma)$-directed.  Each is assumed to be  locally left noetherian.
\end{boxcode}
  \bigskip
  
\subsection{}  Here, we introduce several homological concepts which are to be used throughout  the paper.  
 In particular, we will find the appropriate analogues for directed algebras of a locally free sheaf and of the dual 
functor  $ \Hom_A(-,A)$  over a unital ring $A$.

\subsection{Definition.} \index{Ext groups for $S\lQGr$}  \label{ext-defn} (i)
Let ${\sf N} \in (S\otimes T)\lGr$ and   $\gamma\in\Gamma$. 
Mimicking \eqref{projgen},    set
 $${\sf N}_{\ast, \gamma} = \bigoplus_{\lambda \in \Lambdaplus} {\sf N}_{\lambda, \gamma}\in S\lGr.$$

(ii) For  $i\geq 0$, define a  functor $\SSExt^{\,^i}_{S\lQGr} ( - ,{\sf N}):  S\lQGr \to T\lQGr$  by 
\begin{equation}\label{ext-defn2}
\SSExt^{\,^i}_{S\lQGr} (\mathcal{M} ,{\sf N}) \ = \  \pi_T\bigg(
\,   \bigoplus_{\gamma\in \Gammaplus} 
\Ext^i_{S\lQGr}\bigl(\mathcal{M}, \pi_S ({\sf N}_{\ast,\gamma}) \bigr)\bigg)
\qquad \text{ for 
$\mathcal{M}\in S\lQGr$} .\end{equation}

To see that part (ii) of  this definition makes sense,    use  \ref{hommumjum} to   pick  an injective resolution 
${\sf I}^{\bullet}$ of  ${\sf N}_{\ast,\gamma}$. By \cite[Corollaire~2, p.375]{Ga} $\pi_S({\sf I}^{\bullet})$ is an injective
 resolution of $\pi_S ({\sf N}_{\ast,\gamma})$ and so this can be used to calculate the Ext-groups in question. Moreover for any 
 $t\in T_{\gamma_1, \gamma_2}$ left multiplication by $t$ induces  a morphism 
 $t\,\cdot : {\sf N}_{\ast,\gamma_2} \to {\sf N}_{\ast,\gamma_1}$   in $S\lGr$ 
 and hence  in $S\lQGr$. This  morphism  provides the direct sum  in \eqref{ext-defn2} 
 with the required  structure  of a graded left $T$-module.

\subsection{} \label{agreement}
We can modify the above definition, replacing ${\sf N}\in S\otimes T\lGr$ with $\mathcal{N}\in S\otimes T\lQGr$, 
by setting    \index{Ext groups for $S\lQGr$} 
 \begin{equation}\label{ext-defn3}\SSExt^{\,^i}_{S\lQGr} (\mathcal M, \mathcal{N}) \  = \ 
 \SSExt^{\,^i}_{S\lQGr} ( \mathcal M, \sigma_{S\otimes T}(\mathcal{N}) )
   \qquad\text{for\ }  \mathcal M \in S\lQGr.\end{equation}
  In this 
definition note that the section functor $\sigma_{S\otimes T}: S\otimes T\lQGr \rightarrow S\otimes T \lGr$ is not necessarily right exact, just as proper pushforward for quasi-coherent sheaves is not. 

We need to check that this definition coincides with \eqref{ext-defn2} when $\mathcal{N}=\pi_{S\otimes T}(N)$, and 
$\mathcal M$ is noetherian.

\begin{lemma} Assume that $\mathcal{M}\in S\lqgr$. Then for all $i\geq 0$ and   ${\sf N}\in S\otimes T\lGr$ there exists 
 a natural   isomorphism in $N$ $$
\SSExt^{\,^i}_{S\lQGr} (\mathcal M, {\sf N})\ \cong \  \SSExt^{\,^i}_{S\lQGr} (\mathcal M, \pi_{S\otimes T}({\sf N})). $$
\end{lemma}
\begin{proof}
Let ${\sf Z}\in S\otimes T\lGr$ be torsion; thus  ${\sf Z} = \displaystyle \lim_{\rightarrow} {\sf Z}(j)$ where each 
${\sf Z}(j)\in S\otimes T\lgr$ has the property that there exists $(\lambda_j, \gamma_j)\in \Lambda\times \Gamma$ such that 
  ${\sf Z}(j)_{\alpha,\beta} = 0$  for all $\alpha\geq \lambda_j$ and $\beta\geq \gamma_j$. 
The  first step of the proof will be to show that 
\begin{equation} \label{vanhom} 
\SSExt^{\,^i}_{S\lQGr} (\mathcal{M} ,{\sf Z}) = 0
\qquad \text{ for any $i\geq 0$. }
\end{equation}
Well,
\begin{eqnarray*}    \SSExt^{\,^i}_{S\lQGr} (\mathcal{M} ,{\sf Z}) & \cong & \pi_T \bigg(
\,   \bigoplus_{\gamma\in \Gammaplus} 
\Ext^i_{S\lQGr}(\mathcal{M}, \pi_S \bigl({\sf Z}_{\ast,\gamma}) \bigr) \bigg) 
 \\  \noalign{\vskip 6pt} 
& \cong &
 \pi_T\bigg(
\,   \bigoplus_{\gamma\in \Gammaplus} 
\Ext^i_{S\lQGr}\bigl(\mathcal{M}, \lim_{\rightarrow}\pi_S ( {\sf Z}(j)_{\ast,\gamma}) \bigr)\bigg)
 \\  \noalign{\vskip 6pt} 
& \cong &
\pi_T \bigg(\, \,   \bigoplus_{\gamma\in \Gammaplus}  \lim_{\rightarrow}
\Ext^i_{S\lQGr}\bigl(\mathcal{M}, \pi_S ( {\sf Z}(j)_{\ast,\gamma}) \bigr)\bigg)
 \\  \noalign{\vskip 6pt} 
& \cong &
 \lim_{\rightarrow} \pi_T\bigg( \, \,   \bigoplus_{\gamma\in \Gammaplus} 
\Ext^i_{S\lQGr}\bigl(\mathcal{M}, \pi_S ( {\sf Z}(j)_{\ast,\gamma}) \bigr)\bigg).
\end{eqnarray*}
 Here the second and the last isomorphism hold since $\pi_S$ and $\pi_T$ commute with direct limits, \cite[p.378--9]{Ga},
 while the third isomorphism holds because $\mathcal{M}$ is noetherian. But for any $\gamma \geq \gamma_j$, the object
   ${\sf Z}(j)_{\ast,\gamma}$ is $S$-torsion since $({\sf Z}(j)_{\ast,\gamma})_{\lambda} = {\sf Z}(j)_{\lambda,\gamma} = 0$ for
    all $\lambda\geq \lambda_j$. Hence $\pi_S({\sf Z}(j)_{\ast,\gamma}) = 0$ and  the   $T$-module 
$\bigoplus_{\gamma\in \Gammaplus} 
\Ext^i_{S\lQGr}(\mathcal{M}, \pi_S ( {\sf Z}(j)_{\ast, \gamma}) )$ is torsion. Thus $\pi_T$ kills each such module  and 
  \eqref{vanhom}  is proven.

 Set ${\sf N}' = \sigma_{S\otimes T} (\pi_{S\otimes T}({\sf N}))$.  By  \cite[Proposition~3(2), p.371]{Ga}
   there   are graded $S\otimes T$-modules 
 ${\sf M}\subseteq {\sf N}$ and ${\sf M}'\subseteq {\sf N}'$ 
such that  (i) ${\sf M}$   and  
 ${\sf N}'/{\sf M}'$ are torsion,  and (ii) there exists
   an isomorphism $\psi: {\sf N}/{\sf M} \stackrel{\sim}{\longrightarrow} {\sf M}'$.
  Now apply the functor $\SSExt^{\,^i}_{S\lQGr} (\mathcal{{M}} , - )$
 to the  short exact sequences  
 $$
 0 \longrightarrow {\sf M}\longrightarrow {\sf N} \longrightarrow {\sf N}/{\sf M} \longrightarrow 0\quad \text{and}
  \quad 0\longrightarrow {\sf M}'\longrightarrow {\sf N}' \longrightarrow {\sf N}'/{\sf M}' \longrightarrow 0.
  $$
 Using  the fact that  $\pi_S$ and $\pi_T$ are exact, we get long exact sequences 
$$
\cdots \rightarrow \SSExt^{\,^i}_{S\lQGr} (\mathcal{M} , {\sf M} ) \rightarrow \SSExt^{\,^i}_{S\lQGr} (\mathcal{M} , {\sf N} )
 \rightarrow \SSExt^{\,^i}_{S\lQGr} (\mathcal{M} , {\sf N}/{\sf M}) \rightarrow \cdots 
 $$ 
 and 
 $$
 \cdots \rightarrow \SSExt^{\,^i}_{S\lQGr} (\mathcal{M} , {\sf M}' ) \rightarrow \SSExt^{\,^i}_{S\lQGr} (\mathcal{M} , {\sf N}' ) 
 \rightarrow \SSExt^{\,^i}_{S\lQGr} (\mathcal{M} , {\sf N}'/{\sf M}' ) \rightarrow \cdots .
 $$  
By \eqref{vanhom}, it follows that 
 $$
 \SSExt^{\,^i}_{S\lQGr} (\mathcal{M} , {\sf N}) \cong \SSExt^{\,^i}_{S\lQGr} (\mathcal{M} , {\sf N}/{\sf M} ) \quad \text{and} 
 \quad \SSExt^{\,^i}_{S\lQGr} (\mathcal{M} , {\sf M}') \cong \SSExt^{\,^i}_{S\lQGr} (\mathcal{M} , {\sf N}').
 $$ 
But   $\psi$ induces an isomorphism 
 $\SSExt^{\,^i}_{S\lQGr} (\mathcal{M} ,{\sf N}/ {\sf M} ) \cong \SSExt^{\,^i}_{S\lQGr} (\mathcal{M} , {\sf M}').$
Combining these  three isomorphisms gives the desired result. \end{proof}

\subsection{Acyclic sheaves}\label{acyclic-sect}
Assume that  $S_{00} = T_{00}^{\opp}$. We  define   
\begin{equation}\label{acyclic-defn} \index{$\mathcal{X}$, the dualizing sheaf}
\mathcal{X}_{S\otimes T}\ =\  \pi_{S\otimes T}({\sf S}_{\ast,0}\otimes_{S_{00}}  {\sf T}_{\ast 0} )
 =   \pi_{S\otimes T} \Bigl({\bigoplus}_{\lambda\in\Lambda, \gamma\in \Gamma} S_{\lambda , 0}\otimes_{S_{00}} 
 T_{\gamma, 0}\Bigr)\ \in \ S\otimes T\lQGr
 \end{equation} 
 and, for $\gamma \in \Gamma$, set
 \begin{equation}\label{acyclic-defn2}
 \XXgamma\ = \ \XX_{\ast,\gamma}\ =\ \pi_S(\bigoplus_{\nu}S_{\nu,0}\otimes_{S_{0,0}} T_{\gamma,0})\ \in\ S\lQGr.
 \end{equation}
  We will usually drop the subscript  from $\mathcal{X}_{S\otimes T} $  as it will be clear from the context.
 
  \begin{definition}\index{Acyclic sheaf in $S\lQGr$}
A module $\MM\in S\lqgr$ is called an \emph{acyclic sheaf}, or more strictly an $(S,T)$-acyclic sheaf,  if 
$\SSExt^{\,^i}_{S\lQGr}(\MM,\,\XX_{S\otimes T})=0$ for all $i>0$.   
 \end{definition}
 
 The significance  of this definition is  that the   dual 
object $\Hom_{S\lQGr}(\mathcal M,\pi_S(S))$ does not behave well over a lower $\Lambdaplus$-directed algebra $S$. 
Indeed,   \ref{lackofsymmetry2}  implies that   $S$  is a torsion   right $S$-module and so $\pi_S(S_S)=0.$
 This makes    $\Hom_{S\lQGr}(\mathcal M,\pi_S(S))$  useless 
in applications. As  will be seen, an appropriate replacement  is  the $T$-module $\mathcal{HOM}_{S\lQGr}(\MM,\XX_{S\otimes T})$, for a suitable directed algebra $T$. 
  
Now suppose that  $S = T = \tR$ for  a commutative $\Lambdaplus$-graded algebra $R$ for which   $X = \Proj(R)$ is smooth.
If we identify  $S\lqgr=\coh(X)$, then $\mathcal{X}$ is just
 $\iota_{\ast} \mathcal{O}_X$ under  the diagonal embedding  $\iota: X\longrightarrow X\times X$. 
Since $\mathcal{E}xt^j_X(\mathcal{M}, \mathcal{O}_X)$
 can be calculated on affine patches, it follows that in this case  an acyclic sheaf is just a vector bundle on $X$.

\subsection{Proposition}   
\label{projective}{\it  Assume  that
 $S_{00} = T_{00}^{\opp}$. Let $\lambda\in \Lambdaplus$ 
be a good  parameter for  $S$ and $\gamma\in\Gamma$  a good  parameter for   $T$.
\begin{enumerate}
\item
 $\mathscr{S}_{\lambda}=
 \pi_S({\sf S}_{\ast, \lambda}) $ is a projective object in $S \lQGr $. In particular, for all
 $\mathcal{N}\in S\lQGr$ and $i> 0$ we have 
$\Ext_{ S \lQGr }^i(\mathscr{S}_{\lambda},\, \mathcal{N})=0$  and so $\mathscr{S}_{\lambda}$ is an acyclic sheaf.

\item  
 $\XXgamma  \in S\lqgr$  for any $\gamma\in \Lambdaplus$. Consequently, 
  $\mathcal{X}$ is  noetherian as an object in  $(S\otimes T)\lQGr$.  
  \end{enumerate}}
\smallskip

\begin{proof} (1)
Under the equivalence of categories $S_{\lambda,\lambda}\lMod\buildrel{\sim}\over{\longrightarrow}  S_{\geq \lambda} \lQGr $
defined in the proof of Proposition~\ref{corner-ring}, the projective  object $S_{\lambda, \lambda}$ is sent to 
$\bigoplus_{\tau} S_{\tau +\lambda ,\lambda}.$ On shifting
 by $(-\lambda)$ we then get the object whose $\mu$th component is $S_{\mu,\lambda}$ if $\mu \geq \lambda$ and is
  zero otherwise. But on applying $\pi_S$ this is $\mathscr{S}_{\lambda}$.
  
  (2) By  Proposition~\ref{corner-ring}, again,  the  noetherianity  of  $ \XXgamma $  is equivalent
   to the noetherianity of the $S_{\lambda, \lambda}$-module $S_{\lambda,0}\otimes_{S_{0,0}} T_{\gamma,0}$. This is 
   noetherian since $T_{\gamma,0}$ is a finitely generated module over  $S_{0,0}=T_{0,0}^{\opp}$  and $S_{\lambda,0}$ is a finitely generated 
 left   $S_{\lambda,\lambda}$-module.   
 
  In order to prove that  that $\XX\in (S\otimes T)\lqgr$,  recall  that 
  $T_{\mu,\gamma}$ is a progenerator for all $\mu\geq  \gamma$  and hence 
   $\XXmu =  \pi_S(\bigoplus_{\alpha}S_{\alpha,0}\otimes_{S_{0,0}} T_{\mu,0})
   = \pi_S\bigl( T_{\mu,\gamma} \otimes_{T_{\gamma,\gamma}}\XXgamma\bigr)  $
    for all such $\mu$. Therefore, 
    $$\XX\  =\ \pi_{S\otimes T} \bigl( \bigoplus_\nu \XXnu \bigr)
    \  =\ \pi_{S\otimes T} \bigl( \bigoplus_{\mu\geq \gamma}  \XXmu \bigr)
    \ =\  \pi_{S\otimes T}  \bigl(T_{\geq \gamma} \,\XXgamma\bigr).
   $$
   By the above paragraph this is a noetherian object in $ S\otimes T\lQGr$. 
   \end{proof}
 
Further acyclic sheaves are provided by Theorem~\ref{S-vee2}.  
 

\section{Filtrations.}\label{filtration-section}

\subsection{}    In this section we discuss the basic properties of filtrations on directed algebras. 

\begin{notation}\label{filt-notation}
An  $\mathbb{N}$-filtration $F^{\bullet}S$ on a lower $\Lambda$-directed algebra $S$  will always be assumed 
to respect the graded structure of $S$. Thus $F^m(S)=\bigoplus_{\nu\geq \lambda } F^mS_{\nu,\lambda}$ for all $m\geq 0$.
 This ensures that the associated graded object $\gr_FS$ has an induced lower $\Lambdaplus$-directed algebra 
structure.

Suppose that $A=\bigcup F^iA$ is a filtered right  module
 and that $B=\bigcup F^iB$ is a filtered left  module over some algebra $U$. Then the \emph{tensor product filtration}\index{Tensor product filtration}
on the vector space $A\otimes_UB$ is defined by
$F^m(A\otimes_U B)=\sum_j F^jA\otimes_U F^{m-j}B$ for $m\in \mathbb N$.  Note that for a filtered algebra $U=\bigcup F^\bullet U$, 
the multiplication map $\mu:U\otimes U\to U$ is   automatically filtration preserving in the sense that 
$\mu\bigl(F^m(U\otimes U)\bigr)\subseteq F^mU$ when the left hand side is given the tensor product filtration.
\end{notation}
 
\subsection{Hypotheses}    \label{filthyp} 
Given  a lower $\Lambdaplus$-directed algebra $S$ with an $\mathbb{N}$-filtration $F^{\bullet}S$, we assume
\begin{enumerate}
\item[(H1)] There exists a good parameter $\lambda_S$ for $S$;  
\item[(H2)]  As $\Lambdaplus$-directed algebras, $\gr S\cong \tR$, where 
$R= \bigoplus_{\lambda\in\Lambdaplus} R_\lambda$ is a finitely 
generated commutative $\Lambdaplus$-graded domain
 satisfying $R_{\nu}\cdot R_{\nu'} = R_{\nu+\nu'} $  for all $ \nu, \nu'\in \Lambdaplus.$
\end{enumerate}
Given a second  $\Lambdaplus$-directed algebra $T$ satisfying $T_{00}^{\opp} = S_{00}$, we also assume
\begin{enumerate} 
\item[(H3)] Hypotheses (H1) and (H2) also hold for  $T$ with $\lambda_S = \lambda_T$. Furthermore, under the
 tensor product filtration,      
  $ \gr (S_{\tau,0}\otimes_{S_{0,0}} T_{\gamma,0}) = R_{\tau +\gamma}$  
   for each $\tau, \gamma\geq \lambda_S$.
  \end{enumerate} 
\medskip

Hypothesis (H2) ensures that $S$  is locally left noetherian by the comments following Definition~\ref{projgen1}. 
Similarly, the first part of Hypothesis (H3) ensures that   $T$ and  $S\otimes T$ are locally left noetherian. 
Therefore the underlying assumptions of Section~\ref{homtech} follow from these hypotheses. 

The final part of Hypothesis (H2)  
guarantees that the $n^{\text{th}}$ graded piece 
$ \gr_n( S_{\nu , \tau}) $ of $\gr S_{\nu,\tau}$  satisfies 
$ \gr_n( S_{\nu , \tau}) =\sum_{j\leq n} \gr_j(S_{\nu , \rho} )\gr_{n-j}(S_{\rho ,\tau})$
for each $n$ and $\nu\geq\rho\geq\tau\in \Lambdaplus$.
 Thus (H2) implies that  the multiplication map  $\mu$ is \emph{filtered-surjective} in the sense that under the tensor product filtration 
\begin{equation}\label{filthyp22}
 \mu(F^m(S_{\nu , \rho} \otimes_{S_{\rho,\rho}} S_{\rho ,\tau}))  \ = \ F^m( S_{\nu , \tau}) \qquad
\text{for each $m$ and $\nu\geq\rho\geq\tau\in \Lambdaplus.$}
\end{equation}

\subsection{} \label{sheaf-filt-defn}
Here are our conventions concerning filtrations on objects from  $\ttt{Qgr}$.  
\begin{definition}\index{A (good) filtration on $M\in S\lQGr$}
{\it A filtration on 
$\MM\in S\lQGr$} is a pair 
$({\sf M}, F^{\bullet})$ consisting of choice of a lift ${\sf M}\in S\lGr$ of $\MM$ and a $\mathbb{Z}$-filtration 
$F^{\bullet}$ on ${\sf M}$ 
that is separated and exhaustive. 
We typically 
abuse notation and denote such data by $F^{\bullet}\mathcal{M}$. A filtration on $\mathcal{M}$ gives rise to the 
quasi-coherent sheaf $\gr_F \mathcal{M} = \pi_{_{\tR}}(\gr {\sf M})$ on $X$.  A \emph{filtration preserving morphism} 
between $F^{\bullet}\MM$ and 
 $F^{\bullet}\mathcal{N}$ is just a filtration preserving morphism in $S\lGr$ between the given lifts ${\sf M}$ and ${\sf N}$.
It induces a morphism $\gr_F \MM \rightarrow \gr_F \mathcal{N}$ in $R\lQGr$. A {\it good filtration} on $\MM\in S\lQGr$
 is a filtration such that $\gr_F \MM \in R\lqgr$.
\end{definition}
Note that if $G^{\bullet}\mathcal{M}$ is another filtration of $\MM$ that agrees with $F^{\bullet}\MM$ in high degree, 
 then $\gr_F \MM = \gr_{G} \MM$. The existence of a good filtration implies that $\MM\in S\lqgr$; conversely, any object 
 of $S\lqgr$ can be given a good filtration, thanks to 
the combination of Proposition~\ref{corner-ring} and \cite[Lemma~2.5]{GS2}.

 \subsection{}\label{resolution-chat}
Assume that  $S$ satisfies Hypotheses (H1) and (H2) and write $X=\Proj R$. For $\lambda\geq \lambda_S$,
 recall the  module $\mathscr{S}_\lambda=\pi_S({\sf S}_{\ast, \lambda})$ from
Corollary~\ref{projective}.  Then  the given filtration $F^\bullet$  on $S$   provides induced filtrations, again written $F^\bullet$,
 on ${\sf S}_{\ast, \lambda}$ and 
$\mathscr{S}_\lambda$.   We have 
$\gr_F \mathscr{S}_{\lambda} \cong \pi_R(R[-\lambda]) = \mathcal{O}_X(-\lambda)$ as objects in $\tR\lqgr$.
 
\begin{definition}Let  $\MM\in S\lQGr$ have a good filtration $F^{\bullet}\mathcal{M}$. We will say that an exact sequence 
$\mathcal{P}^\bullet\to \MM\to 0$ is {\it a filtered projective resolution of $\MM$ in $S\lQGr$}\index{Filtered projective resolutions in $S\lQGr$}
 if the following hold.
\begin{enumerate}
\item For some  $\tau\geq \lambda_S$  each $\mathcal{P}^r$ can be written 
  $\mathcal{P}^r =\bigoplus_{1\leq j\leq n_r} \mathscr{S}_{\tau}\epsilon_{jr}\cong \mathscr{S}_\tau^{(n_r)}$ 
for some basis $\{\epsilon_{jr}\}$.
\item Give each ${\sf S}_{\ast,\tau}$ the filtration induced from $S$, and give  ${\sf S}_{\ast,\tau}\epsilon_{jr}$ 
 the induced filtration $G^\bullet$ by assigning $\epsilon_{jr} \in G_{k_{jr}}$ for some $k_{jr}$.
Then the resolution  $\mathcal{P}^\bullet\to \MM$ 
is filtered and the associated graded complex $\gr_G \mathcal{P}
^\bullet\to\gr_F\MM\to 0$ is 
exact.  
\end{enumerate}
\end{definition} 
 Note that in the second part of the definition
$\gr_G \mathcal{P}^r =\bigoplus_j \mathcal{O}_X(-\tau)\epsilon_{jr}$.
\subsection{}
We can form filtered projective resolutions in $S\lqgr$.

\begin{lemma}\label{resolutions}
Let $S$ satisfy Hypotheses (H1) and (H2)  and suppose that  $\MM\in S\lQGr$ has a good filtration
$({\sf M},F^\bullet)$. Then $\MM$ has a filtered projective resolution.
\end{lemma}

\begin{proof} We will reduce the  problem to a case where we can apply \cite[Lemma~2.5(i)]{GS2}.   
It does no harm to replace   ${\sf M}$ 
by ${\sf M}_{\geq \tau}$, for any $\tau \geq \lambda_S$, so we do. Now, by
Proposition~\ref{corner-ring}, $S\lQGr   \simeq S_{\tau, \tau}\lMod$ and $S_{\geq \tau}$ is a Morita $\Lambdaplus$-algebra.

Let $\gr_F {\sf M}$ be generated by   $\sum_{i=1}^n\gr_F {\sf M}_{\tau_i}$ and pick
  $\tau \geq \lambda_S , \tau_1, \ldots , \tau_n$.   By (H2)    $\gr_F{\sf M}_{\geq \tau}$ is generated by $\gr {\sf M}_\tau,$  so
replacing $M$ by ${\sf M}_{\geq \tau}$   means that ${\sf M}$ is now  generated by $N={\sf M}_\tau$. 
The multiplication map 
$\mu_\nu: S_{\nu, \tau}\otimes_{S_{\tau, \tau}} N \rightarrow {\sf M}_{\nu}$ is therefore  surjective for all $\nu\geq \tau$. 
If it is not an  isomorphism for some $\nu$ then Ker$\,\mu_{\nu}\not=0$. But as each $S_{\alpha,\nu}$ is  
a progenerator, this would imply that 
 Ker$\,\mu_{\alpha}\not=0$ for all $\alpha\geq \nu$,     contradicting  Proposition~\ref{corner-ring}.  Hence each $\mu_\nu$ is 
 an isomorphism.

Let $F^{\bullet}N$ denote the induced filtration on $N$.  We claim that 
each $\mu_\nu$  is   a filtered isomorphism, where the domain is given the tensor product filtration $T^\bullet$ induced from
 the filtrations on $S_{\nu ,\tau}$ and  $N$ while the range is given the filtration $F^\bullet$ induced from that on ${\sf M}$.
To see this, note that  $\mu_\nu$ is  filtered by construction  and so  it is   a filtered injection. That it is a filtered surjection 
for all $\nu$  is equivalent to the fact that $\gr_F{\sf M}$ is generated by $\gr_FN$. This proves the claim.
To summarise,   replacing ${\sf M}$ by some ${\sf M}_{\geq \tau}$, we can  assume that its filtration is induced from that on  
$N={\sf M}_{\tau}$.

Set $U=S_{\tau, \tau}$ and apply the proof of  \cite[Lemma~2.5(i)]{GS2}. This constructs a free $U$-module 
$Q=\bigoplus U\epsilon_i$,   with a filtration $G^\bullet$ 
defined  by giving each $\epsilon_i$ some  degree $k_i\geq 0$, for which there exists  a filtered surjection 
$\phi: Q\twoheadrightarrow N$.  Set     $P(\nu)=S_{\nu,\tau}\otimes_UQ$ for $\nu\geq\tau$ and  
${\sf P} = {\sf S}_{\ast,\tau}\otimes_{U} Q = \bigoplus_{\nu\geq \tau} P(\nu)$
with the tensor product filtration $T^\bullet$.
As in [loc.cit] it follows that  
 $\gr_T {\sf P}\cong \bigoplus R\epsilon_i$ is a  
free $R$-module such that the induced map 
$\gr\phi : \gr_T{\sf P}\to \gr_F {\sf M}$ is also surjective.

Now repeat this procedure for $N'=\ker(\phi)$,  with the 
filtration induced from that of $Q$, and iteratively obtain a filtered resolution $Q^\bullet\to N\to 0$ 
that satisfies all  the usual properties of
filtered free resolutions of modules.  Note that, as we are working with 
$S_{\tau,\tau}$-modules here, the choice of $\tau$ stays fixed throughout this procedure.
Finally, using  Proposition~\ref{corner-ring}, one sees  that  
 ${\sf S}_{\ast, \tau}\otimes_{S_{\tau,\tau}}Q^\bullet\to {\sf S}_{\ast, \tau}\otimes_{S_{\tau,\tau}}{\sf M}_{\tau}\to 0$ is then the 
 desired filtered projective  resolution of ${\sf M}$.
\end{proof}

\subsection{} \label{resolutions2}
  {\bf Corollary.}  {\it 
Let $S$ and $T$  satisfy Hypotheses (H1--H3) and pick $\MM\in S\lqgr$ and $\mathcal{N}\in S\lQGr$.

\begin{enumerate}
\item The groups 
$\Ext^\bullet_{S\lQGr}(\MM,\mathcal{N})$    
can be computed as the cohomology groups 
 $\mathrm{RHom}_{S\lQGr}(\mathcal{P}^\bullet, \mathcal{N})$ for   any
filtered projective resolution $\mathcal{P}^\bullet\to\MM\to 0$. 

\item Define $\mathcal{X}
=\mathcal{X}_{S\otimes T}$ by  \eqref{acyclic-defn}.    
Then $$\SSExt_{S\lQGr}^\bullet(\MM,\,\XX) = \pi_T\Bigl(\bigoplus_{\lambda\in\Lambdaplus} 
\mathrm{RHom}_{S\lQGr}\bigl(\mathcal{P}^\bullet, \pi_S({\sf S}_{\ast, 0}\otimes_{S_{0,0}} T_{\lambda ,0})\bigr)\Bigr).$$
\end{enumerate}
}

\begin{proof}  (1)  Since everything can be computed in 
$S_{\tau,\tau}\lMod$ for  $\tau\gg 0$, this follows from Lemma~\ref{resolutions}. 

  (2) As the members of the projective resolution of $\mathcal{M}$ are noetherian, this follows from (1) combined with 
  Lemma~\ref{agreement}.
 \end{proof}


\section{Dualising and associated graded modules.} \label{section-dualising}
\medskip

\begin{boxcode} Assume throughout this section that $S$ and  $T$ are filtered, locally left noetherian $\Lambdaplus$-directed algebras that 
satisfy (H1--H3) from \ref{filthyp}. Fix objects $\MM,\NN \in S\lqgr$ with  good filtrations $({\sf M}, F^{\bullet})$, respectively
 $({\sf N}, G^{\bullet})$,  where  
 ${\sf M}, {\sf N}\in S\lgr$.  Set $X=\Proj(R)$.
\end{boxcode}

\bigskip

\subsection{}  

In this section    results on filtrations and associated graded modules for unital algebras are
 generalised to  directed algebras. This will culminate in a spectral sequence for the associated graded 
groups of Ext groups; see Proposition~\ref{spectral1} for the details.

\subsection{}  
Recall from \cite[p.365]{Ga} that  
 $\Hom_{S\lQGr}(\mathcal{M}, \mathcal{N}) =  
\displaystyle \lim_{\rightarrow} \, \Hom_{S\lGr} ({\sf M}' , {\sf N}/{\sf N}')$, where  ${\sf M}'\subseteq {\sf M}$ 
and ${\sf N}'\subseteq {\sf N}$ are submodules  for which 
${\sf M}/ {\sf M}'$   and ${\sf N}'$ are torsion. Since  ${\sf M}$ and ${\sf N}$ are both noetherian,  
  $({\sf N}/{\sf N}')_{\nu} = {\sf N}_{\nu}$ and ${\sf M}'_{\nu} = {\sf M}_{\nu}$ for $\nu\gg 0$ and so 
   any  $\theta\in \Hom_{S\lQGr}(\mathcal{M}, \mathcal{N})$ belongs to  
  $\Hom_{S\lGr}( {\sf M}_{\geq\nu} , {\sf N}_{\geq \nu})$ for $\nu\gg 0$. We therefore obtain a well-defined  filtration  
  $\Psi^{\bullet}$ on $\Hom_{S\lQGr}(\mathcal{M}, \mathcal{N})$   by 
\begin{equation}\label{homfil3}
\begin{array}{rl}
\Psi^m\ = \ \Psi^m \Hom_{S\lQGr}(\mathcal{M}, \mathcal{N})\  
=&  \{ \theta \in \Hom_{S\lGr}( {\sf M}_{\geq\nu} , {\sf N}_{\geq \nu})  \text{ such that }  \\
\noalign{\vskip 5pt}
& \qquad \theta(F^t{\sf M}_{\geq \nu}) \subseteq G^{t+m}{\sf N}_{\geq \nu}    \text{ for all } t\in\mathbb{Z} \text{ and  $\nu\gg 0$}\}.
 \end{array}\end{equation}

\subsection{Lemma}\label{homfil2}
  {\it 
The   filtration \eqref{homfil3} is  separated and exhaustive.
}
 
\begin{proof}  
For $\beta>\alpha\in\Lambda$ and $\theta\in \Hom({\sf M}_{\geq \alpha },{\sf N}_{\geq \alpha }) $, we will denote the 
image of $\theta$ in $\Hom({\sf M}_{\geq \beta},{\sf N}_{\geq \beta}) $ by $\theta$ too.
Given   $\alpha \in \Lambdaplus$, define the usual exhaustive filtration $\Psi^\bullet_\alpha $ on 
$\Hom_{S\lGr}({\sf M}_{\geq \alpha }, {\sf N}_{\geq \alpha })$ by 
$$\Psi^m_\alpha  
=\Big\{\theta \in \Hom({\sf M}_{\geq \alpha },{\sf N}_{\geq \alpha }) : \theta(F^t {\sf M}_\beta)\subseteq G^{t+m}{\sf N}_{\beta}
 \text{ for all $\beta \geq \alpha$   and $t\in \mathbb Z$}\Big\}.$$
Since $\Psi^m_\alpha\subseteq \Psi^m_{\beta}\subseteq \Psi^m$, for all $m$ and $\beta\geq \alpha $, 
 it follows that $\Psi$ is exhaustive. Thus it remains to prove  separability.

For $\theta\in \Hom({\sf M}_{\geq \alpha },{\sf N}_{\geq \alpha }) $ write 
$m_\alpha (\theta) = \min\{m : \theta\in \Psi_\alpha ^m\}$.
Since $G^\bullet $ is good, there exists $\alpha \in\Lambda$ such that  
$\gr_G({\sf N})_{\geq \alpha }= \gr_G({\sf N}_{\geq \alpha })$ 
has   zero torsion submodule;  $\text{Tors}(\gr_G({\sf N}_{\geq \alpha}))=0$. Clearly 
$m_\beta(\theta)\geq  m_{\gamma}(\theta)$  for any 
$\gamma>\beta\geq \alpha $ and  $\theta\in \Hom({\sf M}_{\geq \beta},{\sf N}_{\geq \beta}) $, so the result will follow if 
 $m_\beta(\theta)=  m_{\gamma}(\theta)$ always holds.

Suppose that 
$r=m_\beta(\theta)>m_{\gamma}(\theta)$ for some such  $\gamma, \beta  $ and  $\theta $. Then we can find some 
$\beta'\geq \beta$ 
and  $a\in F^u{\sf M}_{\beta'}$   such that 
$\theta(a)\in G^{u+r}{\sf N}_{\beta'}\smallsetminus G^{u+r-1}{\sf N}_{\beta'}$. Replace $\beta$ by $\beta'$ and $\gamma$ by 
$\gamma'=\max\{\beta',\gamma\}$; noting that we still have  $r=m_\beta(\theta)>m_{\gamma}(\theta)$. 
Clearly  $xa\in F^{u+p}{\sf M}_{\delta}$ for any $x\in F^pS_{\delta,\beta}\smallsetminus F^{p-1}S_{\delta,\beta}$
 with  $\delta\geq\gamma$ and  $p\in \Z$. Since  $m_{\delta}(\theta)\leq m_{\gamma}(\theta)\leq r-1$,  
it follows that  
 $x\theta(a)= \theta(xa) \in G^{u+p+r-1}{\sf N}_{\delta}.$   But this implies that the principal symbol $\sigma\theta(a)$ satisfies 
 $\sigma(x)\sigma\theta(a)= 0$. In other words,  we have $({\tR}_{\delta,\beta})\cdot\sigma(\theta(a)) = 0 $ for any $\delta \geq \gamma$.
Thus $\sigma\theta(a)\in \text{Tors}(\gr_G({\sf N}_{\geq \alpha}))=0$,   contradicting the choice of $\alpha$. 
This means that  $m_\gamma(\theta)=m_{\beta}(\theta)$ for all $\gamma>\beta\geq \alpha$ and hence that  $\Psi$ is separated.
\end{proof}

\subsection{}  
 Any 
$\theta\in \Psi^m \Hom_{S\lQGr} (\MM, \mathcal{N})$ induces a mapping from $F^i {\sf M}_{\nu}/F^{i+1}{\sf M}_{\nu}$ to
 $F^{i+m}{\sf N}_{\nu}/F^{i+1+m}{\sf N}_{\nu}$ for all large enough $\nu$ and hence defines  a homomorphism
  \begin{equation}\label{Xi-defn}
  \Theta_{\mathcal M,\mathcal N}: \gr_{\Psi} \Hom_{S\lQGr} (\MM, \mathcal{N}) \longrightarrow 
  \Hom_{\tR \lQGr} (\gr_F \MM, \gr_F \mathcal{N}) = \Hom_{X} (\gr_F \MM, \gr_F \mathcal{N}).
  \end{equation}
It follows from the definition of the filtration $\Psi$ that  $\Theta_{\mathcal M,\,\mathcal N}$  is injective. It is natural in 
both entries in the category of filtered objects.

\begin{lemma}\label{goodhom}  
Give $\mathscr{S}_{\nu} = \pi_S({\sf S}_{\ast, \nu})$ the filtration induced from $S$. For large enough $\nu$, depending 
on $\mathcal{N}$ and its filtration, the map
 $$ \Theta=\Theta_{\mathscr{S}_\nu,\,\mathcal N}: \gr_{\Psi} \Hom_{S\lQGr} (\mathscr{S}_{\nu}, 
 \mathcal{N}) \longrightarrow \Hom_{\tR \lQGr} (\gr_F \mathscr{S}_{\nu}, \gr_F \mathcal{N}) = 
 H^0(\gr_F \mathcal{N}\otimes \mathcal{O}_X(\nu))$$  
 is an isomorphism.
 \end{lemma}

\begin{proof} 
We have   seen in \ref{resolution-chat}  
that $\gr_F \mathscr{S}_{\nu} \cong \mathcal{O}_X(-\nu)$ as objects in $\tR\lQGr$ so the final equality in
 the displayed equation follows.  Since $\gr_F\mathcal{N}$ is coherent, 
$\gr_F\mathcal N\otimes \mathcal{O}_X(\nu) $ has global sections $\gr_F{\sf N}_\nu$  for all $\nu\gg 0$. 
Given $x\in F^i{\sf N}_{\nu}$ the map  $\chi:1_\nu \mapsto x$  defines an element of 
$\Psi^i \Hom_{S\lQGr}(\mathscr{S}_\nu ,\, \mathcal N)$.
Therefore, for $\nu\gg 0$ we have induced graded homomorphisms
$$\gr_F {\sf N}_{\nu} \ \buildrel{\gr \chi\, \,}\over\longrightarrow \ \gr_{\Psi} \Hom_{S\lQGr} (\mathscr{S}_{\nu}, \mathcal{N})
\ \buildrel{\Theta}\over\longrightarrow\
\Hom_{\tR\lQGr} (\gr_F \mathscr{S}_{\nu}, \gr_F \mathcal{N}) \ \buildrel{\cong}\over \longrightarrow\ \gr_F {\sf N}_{\nu}$$ whose 
composition is the  identity on $\gr_F {\sf N}_{\nu}$. It follows that $\Theta$ is surjective and so,  by the comments after
 \eqref{Xi-defn},  it is  an isomorphism.
\end{proof}

\subsection{\bf Definition.}  \label{homfil} \index{The dual $\mathcal{M}^{\vee}$ of $\mathcal{M}\in S\lQGr$}
Let $\mathcal{X}
=\mathcal{X}_{S\otimes T}$ be as defined in    \eqref{acyclic-defn} and recall $\lambda_S=\lambda_T$ from (H3).  Define
$$\MM^\vee\ =\  \SSExt^{0}_{S\lQGr}(\MM,\,\XX) \
\ = \ \pi_T \bigg(\bigoplus_{\gamma\in \Lambda} \Hom_{S\lQGr}(\MM,\,\XX_{\gamma})\bigg) \  \in\  T\lQGr.$$ 
 Lemma~\ref{agreement} provides an explicit lift of $\MM^\vee$ to $T\lGr$:  
 $\MM^\vee =\pi_T \MM^\ve$ where
$$ {\MM}^{\ve}  \ = \   
 \bigoplus_{\gamma\geq \lambda_T}\Hom_{S\lQGr}(\MM,\,\pi_S(\bigoplus_{\nu}S_{\nu,0}
\otimes_{S_{0,0}} T_{\gamma,0}))  \in T\lGr.$$
Symmetric definitions apply to $\MM\in T\lqgr.$
      
We will always give both $\XX$ and $\XXgamma $  the filtration induced from  
 the tensor product filtration on the summands 
 $S_{\nu,0}\otimes_{S_{0,0}} T_{\gamma,0} $.
Then \eqref{homfil3} induces a  filtration $\Psi^\bullet$ 
on ${\MM}^{\ve}$ by
\begin{equation}\label{4.4.1}
\begin{array}{rl}
\Psi^m{\MM}^{\ve}\ =\  \Big\{(\theta_\gamma)\in {\MM}^{\ve} : \theta_{\gamma} (F^tM_\nu)
\ \subseteq & F^{t+m}(S_{\nu,0}\otimes_{S_{0,0}} T_{\gamma,0}) :  \\ \noalign{\vskip 2pt}
&\text{for all } \  \gamma\geq \lambda_T, \, t\in \Z \text{ and } 
\nu\gg 0\Big\}.
\end{array}
\end{equation}
This induces a  filtration, again written $\Psi^\bullet$,  on $\mathcal{M}^\vee$.
We observe that, although $\XX\not\in S\lqgr$, any $(\theta_\gamma)\in \MM^\ve$ only has finitely many non-zero entries 
 and is therefore
contained in $\Hom (\MM,\mathcal{Y})$ for some $\mathcal{Y}\in S\lqgr.$ Therefore, Lemma~\ref{homfil2}
can be applied to give

\begin{lemma}\label{homfil22} The filtration \eqref{4.4.1} is separated and exhaustive. \qed\end{lemma}

\subsection{} The assumption in (H1) that some large enough $\lambda_S$ is good, rather than all 
$\lambda\in \Lambdaplus$ are good, is necessary in applications. 
However, it does have the disadvantage that that we have less  control over  terms like $S_{\lambda_S,0}$. The 
next lemma will allow us to avoid these problems.  For $\nu\geq \lambda_S$, recall the definition of 
$S^*_{\nu,\,\lambda_S }$ from \ref{good-defn}.

\begin{lemma}\label{S-bar}
\begin{enumerate}
\item  ${\sf S}_{\mu,0}=S_{\mu,\nu}S_{\nu,0}\cong S_{\mu,\nu}
\otimes_{S_{\nu,\nu}}S_{\nu,0}$  for all $\mu\geq\nu\geq   \lambda_S$.
 
\item  For all $\nu\geq\tau\geq\lambda_S$ we have $S_{\tau,0}=S^*_{\nu,\tau}\otimes_{S_{\nu,\nu}} S_{\nu, 0}$.
\end{enumerate}
\end{lemma}

\begin{proof}  The fact that ${\sf S}_{\mu,0}=S_{\mu,\nu}S_{\nu,0}$ follows from  \eqref{filthyp22}, while the
isomorphism follows from the fact that $S_{\mu, \nu}$ is a
  progenerator over $S_{ \nu,\nu}$.

(2)  As $S_{\tau,\nu}$ is a progenerator the assertion is equivalent to   (1).
\end{proof} 

\subsection{}  As we show next, the hom filtration \eqref{4.4.1} on $ \mathscr{S}_{ \tau}^\ve$ has a natural
 interpretation in terms of the tensor product filtration.

\begin{lemma}\label{S-vee}    Fix $\tau\geq \lambda_S=\lambda_T$. 

 (1)  There is an isomorphism  of noetherian $T$-graded modules  
 \begin{equation}\label{S-vee-equ}
\Theta:  \mathscr{S}_{ \tau}^\ve\ \ \buildrel{\cong}\over\longrightarrow\  
\bigoplus_{\gamma\geq \lambda_T}{S}_{ \tau , 0}\otimes_{S_{0,0}}T_{\gamma,0}.
\end{equation}
This isomorphism is a filtered isomorphism provided we give the left-hand side the filtration
\eqref{4.4.1}  induced from the standard filtration on $\mathscr{S}_{ \tau}$ and the right-hand side the tensor product filtration.

 (2)  For $\phi\geq \lambda_T $  we have 
 $$(\mathscr{S}_{\tau}^{\vee})^{\ve} \cong \bigoplus_{\mu\geq \lambda_S} \Hom_{T_{\phi,\phi}}
 ({S}_{\tau,0}\otimes_{S_{0,0}}{T}_{\phi,0}, \,\,{S}_{\mu,0}\otimes_{S_{0,0}}{T}_{\phi,0}). $$ 
 \end{lemma}

 \begin{proof}  (1)
Let $\Phi_{ \tau}: S\lQGr \to S_{ \tau, \tau}\lMod$ be the equivalence of categories given by 
 Proposition~\ref{corner-ring} and Remark~\ref{2.8.1}.  
Then, as right $T$-modules, we have
 \begin{eqnarray*}
  \mathscr{S}_{ \tau}^\ve &=&
   \bigoplus_{\gamma\geq \lambda_T} \Hom_{S\lQGr}\bigl(\mathscr{S}_{ \tau},\,
   \pi_S(\XXgamma )\bigr)
\ \cong \
\bigoplus_{\gamma\geq \lambda_T} \Hom_{S_{ \tau, \tau}\lMod}\bigl(S_{ \tau, \tau},\,
 \Phi_{ \tau}(\pi_S(\XXgamma ))\bigr)
  \\ &\cong &
  \bigoplus_{\gamma\geq \lambda_T}   \Phi_{ \tau} \Bigl(\pi_S\bigl(\bigoplus_{\nu\in\Lambda} S_{\nu,0}\otimes_{S_{0,0}} 
  T_{\gamma, 0}\bigr)\Bigr).
 \end{eqnarray*}
Remark~\ref{2.8.1} and Lemma~\ref{S-bar}(2) imply that, for any $\mu\gg 0$ and $\gamma\geq \lambda_T$ we have
 $$ \Phi_{ \tau} \Bigl(\pi_S\bigl(\bigoplus_{\nu\in\Lambda} S_{\nu,0}\otimes_{S_{0,0}} T_{\gamma, 0}\bigr)\Bigr)  
\ =\ S_{\mu,\tau}^*\otimes_{S_{\mu,\mu}}( S_{\mu,0}\otimes_{S_{0,0}} T_{\gamma, 0})
\ \cong\  S_{\tau, 0}\otimes_{S_{0,0}} T_{\gamma,0}.$$ This proves \eqref{S-vee-equ}.
The  argument of Corollary~\ref{projective}(2) ensures that 
 $\bigoplus_{\gamma\geq \lambda_T}{S}_{ \tau , 0}\otimes_{S_{0,0}}T_{\gamma,0}\in T\lqgr.$

 It remains to examine the filtered structure of $\Theta.$ 
Let  $$x = \sum_\gamma \sum_{i_{\gamma}} x^{i_{\gamma}}_\tau\otimes x^{i_{\gamma}}_{\gamma}  \in Y= \bigoplus_{\gamma\geq \lambda_T} S_{\tau, 0}\otimes_{S_{0,0}} T_{\gamma,0}.$$ Then, for any $\nu\geq \tau$, 
 the computations of the last paragraph show that 
  $\Theta^{-1}(x)\in  \mathscr{S}_{ \tau}^\ve$ is the homomorphism that 
  maps  
 $$\begin{array}{rl}  
  f\in S_{\nu,\tau} \  \mapsto & \displaystyle\sum_{\gamma\geq \lambda_T} \sum_{i_{\gamma}} f \otimes x^{i_{\gamma}}_\tau \otimes x^{i_\gamma}_{\gamma} 
\ \mapsto  \  \sum_{\gamma\geq \lambda_T} \sum_{i_\gamma} f\cdot  x^{i_\gamma}_\tau \otimes x^{i_\gamma}_{\gamma}   \\
 \noalign{\vskip 8pt}
 &  \in \ \  \displaystyle
W_\nu  =  \sum_{\gamma\geq \lambda_T} S_{\nu,\tau}\otimes_{S_{\tau,\tau}}S_{\tau,0}\otimes_{S_{0,0}} T_{\gamma,0}
\ \cong \  Z_\nu  =   \sum_{\gamma\geq \lambda_T} S_{\nu,0}\otimes_{S_{0,0}} T_{\gamma,0} .
\end{array}$$ 
By \eqref{filthyp22}, the tensor product filtration on 
$S_{\nu, \tau}\otimes_{S_{ \tau, \tau}}S_{ \tau, 0}$ agrees under multiplication with the 
given filtration on $S_{\nu, 0}= S_{\nu, \tau}S_{ \tau, 0}$. Hence  the treble tensor product filtration on 
$W_\nu$
agrees with the tensor product filtration on 
$Z_\nu$.  Thus, by the last displayed equation,
if $x\in F^r(Y)\smallsetminus F^{r-1}(Y)$, then 
 $\Theta^{-1}(x)$ maps $F^m(S_{\nu,\tau})$ to  $F^{m+r}(Z_\nu)$.  In order to complete the proof of 
the lemma,  we need to prove that 
  $\Theta^{-1}(x)   \in F^r( \mathscr{S}_{ \tau}^\ve )\smallsetminus F^{r-1}( \mathscr{S}_{ \tau}^\ve )$
  or, equivalently, that there exists $\nu\gg 0$ such that  $xF^m(S_{\nu,\tau}) \not\subseteq  F^{m+r-1}(Z_\nu)$ for some $m\geq 0$.

So, suppose that $xF^m(S_{\nu,\tau}) \subseteq  F^{m+r-1}(Z_\nu)$  for   all $\nu\geq \nu_0\gg 0 $ and all $m\geq 0$.
 Then Hypotheses~(H2) and (H3)
imply that, inside $R$, one has
$\sigma(x)\cdot \gr(\bigoplus_{\nu\geq \nu_0} S_{\nu,\tau}) = 0$   
and hence that   $\sigma(x)\cdot R_{\geq (\nu_0-\tau)} =0$. 
This contradicts the fact that, by Hypothesis (H2), $R$ is a domain. 

(2)  By (1)   
$$(\mathscr{S}_{\tau}^{\vee})^{\ve}  \ \cong \  \bigoplus_{\mu \geq \lambda_S}
 \Hom_{T\lQGr }\Bigl( \pi_T\bigl( \bigoplus_{\gamma\geq \lambda_T}{S}_{ \tau , 0}\otimes_{S_{0,0}}T_{\gamma,0}\bigr),\,
 \pi_T\bigl(\bigoplus_\nu S_{\mu,0}\otimes_{S_{0,0}}T_{\nu,0}\bigr)\Bigr).$$
 Now, as in (1), apply the  equivalence of categories $\Phi_\phi : T\lQGr\to T_{\phi,\phi}\lMod$ defined  by 
 Proposition~\ref{corner-ring} and Remark~\ref{2.8.1} for the algebra $T$.   
\end{proof}

  \subsection{}    Combining Lemma~\ref{S-vee}(1) with 
  Hypothesis~(H3) gives:
  
  \begin{corollary}\label{easier2} Pick $\tau\geq \lambda_S$, and give 
  $\mathscr{S}^\ve_\tau$ the filtration $F^\bullet$ of Lemma~\ref{S-vee}(1). Then 
$$  \gr_F \mathscr{S}^\ve_{\tau} = 
 \bigoplus_{\gamma\geq \lambda_T} \gr_F(S_{\tau,0}
  \otimes_{S_{0,0}} T_{\gamma,0}) =  \bigoplus_{\gamma\geq \lambda_T} R_{\tau+\gamma} 
  $$ 
  as objects in $\tR\lgr.$
  
   In particular $F^\bullet$ is a good filtration on  $\mathscr{S}^\ve_\tau$ and
  $
\gr \mathscr{S}^\vee_{\tau}
=\pi_R(R[\tau])
  = \mathcal{O}_X(\tau).
  $
\qed\end{corollary}

 \subsection{} We next  want to  
 filter  the $T$-module  $\SSExt^{\,^p}_{S\lQGr}(\MM,\,\XX)$, 
  and to relate its  associated graded module  to  the sheaf Ext group
  $\SExt^p_{X}(\gr_F \MM,\,\mathcal{O}_X)\in \qcoh(X)$. 
  For this we introduce the following spectral sequence.
  
  \begin{proposition}\label{spectral1} 
  There is a   convergent spectral sequence in $\tR\lQGr = \qcoh(X)$
  \begin{equation}\label{spectral1-equ}
  \mathscr{S}: \quad  E_1^p = 
  \SExt^p_{X}(\gr_F \MM,\, \mathcal{O}_X) \ \Longrightarrow \    
    \gr_F \SSExt^{\, ^p}_{S \lQGr}(\MM,\, \XX),
   \qquad
   \end{equation}
   where the good filtration 
   $F^\bullet$ of $ \SSExt^{\, ^p}_{S \lQGr}(\MM,\, \XX)$ is  defined in the proof.
  \end{proposition}

\begin{proof}
   The proof will follow the argument in \cite[Chapter~2,  Section~4]{Bj1}. We remark that Bjork's
     result is phrased at the
    level of filtered abelian groups, and this   is general enough to encompass much of the present proof.

Using Lemma~\ref{resolutions},  pick a filtered projective resolution
  $ \mathcal{P}^\bullet\to \MM\to 0$, where 
   each $\mathcal{P}^r= \bigoplus_j \mathscr{S}_\nu \ee$ for some basis elements $\ee$    and some $\nu\geq \lambda_S$.
 Here,  the summands $ \mathscr{S}_{\nu} \ee $ are filtered by giving $\ee$  some  degree $k_{jr}$.   
  Write $$({\sf K}^\bullet, d) = (\mathcal{P}^\bullet)^\ve = 
  \bigoplus_{\gamma\geq \lambda_T} \Hom_{S\lQGr}(\mathcal{P}^\bullet,\, \pi_S(\bigoplus_{\nu\in\Lambda} 
  S_{\nu,0}\otimes_{S_{0,0}}T_{\gamma,0})) $$ 
  for the induced complex with differential $d$. By Lemma~\ref{homfil22} this is  a filtered complex in $T\lGr$. 
  Let  $(\KK^{\bullet} , d)\in T\lQGr$ denote the image of $({\sf K}^\bullet,d) $  under $\pi_T$.  
   By Lemma~\ref{S-vee}(1),
   for each $r$
   \begin{equation}\label{spectral11}
   {\sf K}^r \ = \  \bigoplus_j \mathscr{S}_{\nu}^\ve \evee  \ =
   \   \bigoplus_j   \left(\bigoplus_{\gamma\geq \lambda_T} \,S_{\nu,0}
   \otimes_{S_{0,0}}T_{\gamma,0}\right)\, \evee
   \end{equation}
   where $\deg \evee=-\deg \ee$ and then   the right hand side of \eqref{spectral11} is given the (good)  tensor product filtration $F^\bullet$.

In the notation of \cite[pp.48--51]{Bj1} 
our complex ${\sf K}^{\bullet}=\bigoplus {\sf K}^p$ with filtration $F^j$  is Bjork's abelian group $(A,d)$
with filtration $\Gamma_j.$
For each $j\in\mathbb{N}$, set  ${\sf Z}_j^\infty = \Ker(d)\cap F^j$  
and notice that the group  ${\sf H}:=\Ker(d)/\im(d)$ is filtered by
 the $F^j({\sf H})={\sf Z}^\infty_j+\im(d)/\im(d)$
 with associated graded group $\gr {\sf H}=
 \bigoplus_j  \left( {\sf Z}^\infty_j+\im(d)\right)\big/\left({\sf Z}^\infty_{j-1}+\im(d)\right)$. 
By Corollary~\ref{resolutions2},
  \begin{equation}\label{spectral12}
{\sf H} = {\sf H}^\bullet({\sf K}^{\bullet})
 =\bigoplus_p  \bigoplus_{\gamma\geq \lambda_T} \Ext^p_{S\lQGr}( \MM, \pi_S(\bigoplus_{\nu} \,S_{\nu,0}
   \otimes_{S_{0,0}}T_{\gamma,0})).
 \end{equation}  
 
 Now consider the graded group 
 $\gr_F{\sf K}^{\bullet}=\bigoplus F^j/F^{j-1}$  with the 
 induced action $\gr d$ of the differential $d$  and its  cohomology group 
 $\mathbf{H}=\Ker(\gr d)/\im(\gr d)=\bigoplus_p H^p (\gr_F{{\sf K}^{\bullet}})$.
 By  
   \cite[Chapter~2,  Theorem~4.3]{Bj1} and   \eqref{spectral12},   for any 
   $p\geq 0$,  we have a spectral sequence of graded $\tR$-modules:
  \begin{equation*} 
\begin{array}{rl}  \mathscr{S}': \ E_1  =\ \displaystyle  
  H^p (\gr_F{{\sf K}^{\bullet}}) \     
\Longrightarrow \       \gr_{F} \bigg(\displaystyle  \bigoplus_{\gamma\geq \lambda_T} 
  \Ext^p_{S\lQGr}\bigl( \MM,  \pi_S(\bigoplus_{\nu} \,S_{\nu,0}
   \otimes_{S_{0,0}}T_{\gamma,0})\bigr)\bigg)  \quad \end{array}   \end{equation*}

As we noted, $F^\bullet$ is
 a good filtration on each ${\sf K}^r$  and so,  for each $p$,
  $F^q{\sf K}^p=0$  for $q\ll 0$. Therefore, by the proof of 
\cite[Theorem, p.II.15]{Se}, this  sequence converges  if 
for   all $p$ there exists   $s(p) \geq 0$ such that
\begin{equation}\label{serre}
F^{q} {\sf K}^{p+1} \cap d({\sf K}^p) \ \subseteq\  d(F^{q+s(p)} {\sf K}^p)
\qquad \text{for all }   q\in\Z.
\end{equation} 
Since $F^\bullet$ is
 a good filtration on each ${\sf K}^r$,     both
$\left\{{G}_j = F^j{\sf K}^{p+1}\cap d({\sf K}^p)\right\}$
and $\left\{{G}_j'= d(F^j {\sf K}^p)\right\}$  are
 good filtrations on $d({\sf K}^p)$. By \cite [Proposition~8.6.13]{MR}  there therefore  exists $s(p)\geq 0$ such that  
 ${G}_{q}\subseteq {G}'_{q+s(p)}$ for all $q$.
Thus \eqref{serre} holds and  $ \mathscr{S}'$   converges.

Applying the exact functor $\pi_{\tR}$ to  $\mathscr{S}'$ gives a necessarily convergent spectral sequence:
 \begin{equation*}\label{spectral1-equ3}
  \mathscr{S}: \,   E_1^p = 
  H^p (\pi_{_{\tR}}(\gr_F{\sf K}^{\bullet})) \    \Longrightarrow  \ 
     \pi_{_{\tR}}\bigg(\gr_{F} \Bigl(\bigoplus_{\gamma\geq \lambda_T}
   \Ext^p_{S\lQGr}\bigl( \MM,  \pi_S(\bigoplus_{\nu} \,S_{\nu,0}
   \otimes_{S_{0,0}}T_{\gamma,0})\bigr)\Bigr)\bigg)  .
   \end{equation*}
By definition, the right hand side of $\mathscr{S}$  is $\gr_F (\SSExt^{\,^p}_{S \lQGr}(\MM,\, \XX))$
 and so it remains to identify the 
left hand side.
However, by construction, the filtered   resolution $\mathcal{P}^\bullet$
   defines  a projective resolution
     $\gr_F \mathcal{P}^\bullet \to \gr_F\MM \to 0$ of $\gr_F \MM $ in $\tR\lQGr$.
Also,  Hypothesis (H3) ensures that 
 $\gr_{F}\XXgamma$ equals $\mathcal{O}_X(\gamma)$  for $\gamma\gg 0$.        Thus  
   $$     \begin{array} {rcl} 
\pi_{_{\tR}} (\gr_{F} {\sf K}^{\bullet}) &\cong &  \displaystyle
\pi_{_{\tR}} \Bigl( \bigoplus_{\gamma\geq \lambda_T}
 \gr_{\Psi} \Hom_{S\lQGr}(\mathcal{P}^{\bullet}, \XXgamma) \Bigr) \\ 
 \noalign{\vskip 5pt}
 &  \cong &   \displaystyle
 \pi_{_{\tR}} \Bigl( \bigoplus_{\gamma\geq \lambda_T} \Hom_{\tR \lQGr}(  
  \gr_F \mathcal{P}^\bullet,\, \gr_F \XXgamma) \Bigr) \qquad\text{by  Lemma~\ref{goodhom} }
   \\ 
    \noalign{\vskip 5pt}
& \cong &   \displaystyle
 \pi_{_{\tR}} \Bigl(  \bigoplus_{\gamma\geq \lambda_T}\Hom_{\tR \lQGr}(  \gr_F \mathcal{P}^\bullet,\,  
  \mathcal{O}_X(\gamma))\Bigr)\\
   \noalign{\vskip 5pt}
 & \cong & \SHom_{X}( \gr_F \mathcal{P}^\bullet,\,  \OO_X).
\end{array}
$$
Therefore $H^p(\pi_{_{\tR}} (\gr_{F} {\sf K}^{\bullet}) ) = \SExt^p_{X}( \gr_F \MM,\,  \OO_{X}),$ as required.
   \end{proof}

\subsection{}   
The following standard consequence of the spectral sequence \eqref{spectral1-equ} 
will be particularly useful.
  
  \begin{corollary}\label{subquotient} 
  \begin{enumerate}
  \item   
  For each $p\geq 0$ and 
  under the induced grading,  the sheaf \
     $ \gr \SSExt^{\,^p}_{S\lQGr}(\MM,\XX) \in \tR\lQGr$ is a subquotient of 
  $\SExt^p_{X} (\gr_F \MM,\, \mathcal{O}_X)$.
  
  \item
Let  $\MM'\in T\lqgr $ 
  with a good filtration $F$. Then for each $p\geq 0$ and 
  under the induced grading,  
     $ \gr \SSExt^{\,^p}_{T\lQGr }(\MM',\XX) \in \tR\lQGr$ is a subquotient of 
  $\SExt^p_{X} (\gr_F \MM',\, \mathcal{O}_X)$.
\qed 
\end{enumerate}
  \end{corollary}

  \begin{proof}  Part (1)  follows from the proposition combined with  the observation that
 each term $E^{p}_{r+1}$ is a subquotient of $E^p_r$.   
  As  our hypotheses are symmetric in $S$ and $T$, this also proves  (2).
   \end{proof}

\subsection{Remark}\label{gr-comment}
It is well-known that   filtrations that are not good have bad properties. This is especially true
 when filtering 
objects in $S\lqgr$, or even $U\lqgr$ for a graded unitary ring $U$. For an extreme example,
suppose that $F^0S\subseteq \bigoplus_\tau  S_{\tau,\tau}$ and that 
$0\not=\mathcal{M}=\pi_S({\sf M})\in S\lqgr$. Now  
  filter ${\sf M}$ by $F^0({\sf M})= {\sf M}$. Then $\gr_F{\sf M}$ is killed by $\gr_{\geq 1}\tR\supseteq R_{\geq 1}$
   and so $\gr_F{\sf M}$ is torsion. Thus, in the notation of \ref{sheaf-filt-defn},
 $\gr_F\mathcal{M}=\pi_{\tR} (\gr_F({\sf M}))=0$.

Fortunately for good filtrations this problem does not arise as we have 

\begin{lemma} Let $\mathcal{M}\in S\lqgr$ have a good filtration $({\sf M},F^{\bullet})$ and suppose that 
$\gr_F\mathcal{{\sf M}}=0$. Then $\mathcal M=0$. \end{lemma}

\begin{proof}  Since $F^\bullet$ is a good filtration, $\gr_F({\sf M})$ is finitely generated. Since $\pi_{\tR}=0$ 
this means that $\gr_F({\sf M})$ is torsion and so  ${\sf M}$ is also torsion. Hence $\mathcal{M}=0$.
\end{proof}

   
 \section{Commutative theory.}\label{section-commutative}
 \bigskip
 \begin{boxcode}  Fix  a finitely 
 generated commutative $\Lambdaplus$-graded algebra $R = \bigoplus_{\lambda\in\Lambdaplus} R_\lambda$
 and  lower $\Lambdaplus$-directed algebras  $S$ and $T$ that satisfy Hypotheses
  (H1--H3) from \ref{filthyp}. Set $X = \Proj (R)$. 
\end{boxcode}

\bigskip

\subsection{}  We  prove results on  the support of $R$-modules and more specifically the support of 
 $\gr \MM$ for  $\MM\in S\lqgr$.  These  provide analogues for projective varieties of results proven in 
 \cite{Le}  for  affine    varieties and our proofs closely mimic Levasseur's.

 \subsection{\bf Notation} \label{support-defn}
  We will always  identify $X=\Proj(R)$ with the set of graded 
  prime ideals of $R$ that contain no irrelevant ideal.
Given  $\mathfrak{p}\in \Proj (R)$,  write $\mathcal{C}_\mathfrak{p}$ for
  the homogeneous elements in $R\smallsetminus \mathfrak{p}$, with $\Lambdawide$-graded
   localization $R_{\mathfrak{p}} = R\mathcal{C}_\mathfrak{p}^{-1}$ and set $R_{(\mathfrak{p})}$ to be the
   degree zero component $(R_{\mathfrak{p}})_0$.  Thus $R_{(\mathfrak{p})}$ is the local ring of $X$
   corresponding to  the complement of the  subscheme $  \VV(\mathfrak{p})\subset X $. Similarly we write
   ${\sf N}_{(\mathfrak{p})}=({\sf N}\otimes_RR_{\mathfrak{p}})_0$ for  ${\sf N}\in R\lGr$. By
    \cite[Propositions~8.2.2 and 9.1.2]{rob} there is an equivalence  $R_{(\mathfrak{p})}\lMod
    \ \buildrel{\sim}\over\longrightarrow\ R_{\mathfrak{p}}\lGr$ 
    given by $M\mapsto R_{\mathfrak{p}}\otimes_{R_{(\mathfrak{p})}} \hskip -3pt M$, with inverse $L\mapsto L_0 $. 
 
Given $\mathcal{L} \in \tR \lqgr = \coh(X)$, its {\it support} \index{Supp, the support of a coherent sheaf}
$\Supp \mathcal{L}$ is defined to be its 
support in $X$:   $$\Supp \mathcal{L} = \{\mathfrak{p}\in \Proj(R) : \mathcal{L}_\mathfrak{p}\not= 0\}.$$
By the multihomogeneous analogue of \cite[Chapter~II, Proposition~5.11(a)]{Ha}, if $\mathcal{L}=\pi_R{\sf L}$ for   
${\sf L}\in R\lgr$,
 then $\mathcal{L}_{\mathfrak{p}}  = {\sf L}_{(\mathfrak{p})}$. Thus  
 $\Supp \mathcal{L} = \Supp {\sf L} \cap \Proj(R),$
 where $\Supp {\sf L}$ denotes the module-theoretic support of ${\sf L}$.
 
  Combining these observations  with the multi-graded analogue of  \cite[Proposition~III.6.8]{Ha} gives the  
  following description of   the sheaf Ext  groups  of $\mathcal{O}_X$-modules,  denoted by $\SExt_X$.
     
\begin{lemma} \label{localise} 
Let $\mathfrak{p}\in \Proj R$ and ${\sf N}\in R\lgr$ with corresponding coherent sheaf $\NN =\pi_R{\sf N}\in R\lqgr$.
Then $  \Ext^t_{R_{(\mathfrak{p})}}({\sf N}_{(\mathfrak{p})} ,\,R_{(\mathfrak{p})}) 
 \cong\  \SExt^t_{X}(\NN,\, \mathcal{O}_X)_{\mathfrak{p}}$ for 
any $t\geq 0$. \qed
  \end{lemma}

\subsection{}  By \cite[Theorem~(f)]{Ba}. and Lemma~\ref{localise}, 
  $X$ is Gorenstein if and only if  $\SExt^t_{X}(\NN,\, \mathcal{O}_X) = 0 $ for 
all $t>\dim X$ and $\NN\in \coh(X)=R\lqgr$.
Given a  module $M$ over a ring $A$, we write $\ell_A(M)$ for the length of $M$. 
  
\begin{lemma}\label{3.3.3}
Assume that $X$ is Gorenstein.
   Let ${\sf N}\in R\lgr$ with corresponding coherent sheaf $\NN=\pi_R{\sf N}$, and let $t\in\mathbb{N}$. \begin{enumerate}
\item If $\mathfrak{p}\in      \Supp \SExt^t_{X}(\NN,\,\mathcal{O}_X)$, then
     $ht(\mathfrak{p})=\dim R_{(\mathfrak{p})} \geq  t$.

\item Let $\mathfrak{p}$ be a minimal prime in $\Supp \NN$. Then 
  $\Ext^t_{R_{(\mathfrak{p})}}({\sf N}_{(\mathfrak{p})},\,R_{(\mathfrak{p})})=0$ if
  $t\not=\dim\, R_{(\mathfrak{p})}$. If $t=\dim\,  R_{(\mathfrak{p})}$ then 
   $\Ext^t_{R_{(\mathfrak{p})}}({\sf N}_{(\mathfrak{p})},\,R_{(\mathfrak{p})})$ has finite 
  length, equal to  $\ell_{R_{(\mathfrak{p})}}({\sf N}_{(\mathfrak{p})})$.
  
\item $\Supp \SExt^t_{X}(\NN,\,\mathcal{O}_X)=\bigcup_{i=1}^p \VV(\mathfrak{p}_i)\cup \mathcal{D}$, where the
 $\VV(\mathfrak{p}_i)$   run through the irreducible components of $\Supp \NN$ of codimension $t$ in 
   $X$ and $\mathcal{D}$ consists of a
    union of irreducible components of codimension $\geq t$ in $X$, each of 
    which is contained in the union of the irreducible components of $\Supp \NN$ 
     of codimension different from $t$.
 \end{enumerate}
 \end{lemma}
 
   \begin{proof}  (1)  
  If $\mathfrak{p}\in     \Supp \SExt^t_{X}(\NN,\,\mathcal{O}_X)$, then       
  $\Ext^t_{R_{(\mathfrak{p})}}({\sf N}_{(\mathfrak{p})},\, R_{(\mathfrak{p})})\not=0$  by  Lemma~\ref{localise}. 
  Since $R_{(\mathfrak{p})} 
$ is Gorenstein,  \cite[Theorem]{Ba} implies that     
$ht (\mathfrak{p}) = \injdim R_{(\mathfrak{p})} = \dim R_{(\mathfrak{p})} \geq  t$.

    (2) If $\ell_{R_{(\mathfrak p)}}({\sf N}_{(\mathfrak{p})})=\infty $, then   ${\sf N}_{(\mathfrak{q})}\not=0$ for some
     $\mathfrak{q}\subsetneq \mathfrak{p}$, contradicting the minimality of $\mathfrak{p}$ in $\Supp\NN$.
    Hence $\ell_{R_{(\mathfrak{p})}}({\sf N}_{(\mathfrak{p})})<\infty $. Set
     $k_{\mathfrak{p}} = R_{(\mathfrak{p})}/ \mathfrak{p}_{(\mathfrak{p})}$.  Then, as
     $R_{(\mathfrak{p})}$ is   Gorenstein,   \cite[Proposition~2.9]{Ba} implies that 
     $$\Ext^t_{R_{(\mathfrak{p})}} ( k_{\mathfrak{p}}, R_{(\mathfrak{p})}) = 
     \begin{cases} 0 \quad & \text{if $t\neq ht(\mathfrak{p})$} \\ k_{\mathfrak{p}} & \text{if $t =  ht(\mathfrak{p})$.} \end{cases}$$
  Let $L$ be an $R_{(\mathfrak{p})}$-module of finite length.
  By  induction on the  length of $L$,   it follows that  
   $\Ext^t_{R_{(\mathfrak{p})}}(L, R_{(\mathfrak{p})}) = 0$ if $t\not= ht(\mathfrak{p})$, while
   $\ell_{R_{(\mathfrak p)}} (\Ext^{ht(\mathfrak{p})}_{R_{(\mathfrak{p})}}(L, R_{(\mathfrak{p})})) =
   \ell_{R_{(\mathfrak{p})}}(L_{(\mathfrak{p})})<\infty$.  So (2) holds.

    (3)   If $\VV(\mathfrak{p})$  is an irreducible component of
    $\Supp \NN$, then $\mathfrak{p}$ is minimal in $\Supp \NN$. Thus if 
    $t = \codim \VV(\mathfrak{p}) = \injdim  R_{(\mathfrak{p})}$, 
    then part~(2) and    Lemma~\ref{localise} show that  $\mathfrak{p}\in   \Supp \SExt^t_{X}(\NN,\,\mathcal{O}_X)$
    and hence that $\VV(\mathfrak{p})\subseteq \Supp \SExt^t_{X}(\NN,\,\mathcal{O}_X)$.
Conversely, if $\mathfrak{q}\in \Supp  \SExt^t_{X}(\NN,\,\mathcal{O}_X)$, then
$ht(\mathfrak{q})\geq t$ by part~(1).  By    Lemma~\ref{localise},
$0\not= \SExt^t_{X}(\NN,\,\mathcal{O}_X)_\mathfrak{q} \cong 
\Ext^t_{R_{(\mathfrak{q})}}({\sf N}_{(\mathfrak{q})},\,R_{(\mathfrak{q})})$
and so $\mathfrak{q}\in \Supp \NN$.   \end{proof}

 \subsection{\bf Definition}\label{char-defn} \index{$\Char$, the characteristic variety of!an object in $S\lqgr$} 
 Let ${\sf M}\in S\lgr$ and 
pick a good filtration $F^{\bullet}$  for ${\sf M}$.  Following 
 \cite[Definition~2.6]{GS2}, we define the \emph{characteristic
variety of $\sf M$} to be  
$\Char {\sf M} =\Supp (\gr_F{\sf M})\subseteq X$.
The {\it characteristic dimension of $\sf M$} \index{$\chdim$, the characteristic dimension}is defined to be 
 $\chdim(\sf M)=\dim \Char\sf M$. 
 If $M\in S_{\lambda,\lambda}\lmod$, for a good parameter $\lambda$, 
 we write $ \Char(M)=\Char(\Psi(M))$, in the notation of Proposition~\ref{corner-ring}.
 
\medskip 
By the proof of  \cite[Lemma~2.5(3)]{GS2},  $\Char {\sf M} $  is independent of the choice of $F^{\bullet}$.
Moreover, $\Char {\sf N}=\emptyset$ for any torsion module $\sf N$ and so
 $\Char {\sf M}=\Char {\sf M}'$ whenever $\pi_S({\sf M})=\pi_S({\sf M}')$.
 Thus $\Char\MM$  and $\chdim(\MM)$ are  also     well-defined for  $\MM\in S\lqgr$, or indeed for $\MM\in R\lqgr$.  
 
 \subsection{} Recall the cohomology groups  $\SSExt$  
  defined in   \eqref{ext-defn3} for objects in  $S\lQGr$.

\begin{proposition}  \label{3.3.4-extra} Assume that $X$ is Gorenstein. Suppose that  $\MM\in S\lQGr$ has 
a   good filtration $F$ and give  $ \SSExt^{\,^t}_{S\lQGr}(\MM,\XX)$ 
   the induced filtration  $F$, as defined in Proposition~\ref{spectral1}.
  If $\VV(\mathfrak{p})$ is an irreducible component of $\Char \MM$ of 
  codimension  $t$ in $X$,
  then   $$\left(\gr_F  \SSExt^{\,^t}_{S\lQGr}(\MM,\, \XX) \right)_\mathfrak{p} \cong
   \SExt^t_{X} \left( \gr_F(\MM),\, \mathcal{O}_{X}\right)_\mathfrak{p}.$$
   \end{proposition}
   
\begin{proof}  We use the spectral sequence $\mathscr{S}$ from Proposition~\ref{spectral1}.
Here  $E_1^t= \SExt^t_{X}(\gr_F \MM,\,\mathcal{O}_X)$, and by definition  
  $E_m^\ell=H^\ell(E^\bullet_{m-1},d)$  for $\ell\geq 0$. We have
$E^t_m = E^t_\infty =  \gr_F  \SSExt^{\,^t}_{S\lQGr}(\MM,\,\XX)  $ for
 all $m\gg 0$. 
  Clearly,  $\big(E^\ell_r\big)_\mathfrak{p} = H^\ell\big(( E^\bullet_{r-1})_\mathfrak{p}\big)$ and so, by  Lemma~\ref{localise}, 
    $$(E^\ell_1)_\mathfrak{p} =\SExt^\ell_{X}(\gr_F \MM,\,\mathcal{O}_X)_\mathfrak{p}
   \cong \Ext^\ell_{R_{(\mathfrak{p})}}\bigl((\gr_F{\sf M})_{(\mathfrak{p})},\,R_{(\mathfrak{p})}\bigr).$$
  Since $\mathfrak{p}$ is a minimal prime ideal in $\Char \MM = \Supp(\gr_F \MM)$, 
  Lemma~\ref{3.3.3}(2) implies that the only nonzero term in $(E^\bullet_1)_\mathfrak{p}$ is
  $ \Ext^t_{R_{(\mathfrak{p})}}\bigl((\gr_F{\sf M})_{(\mathfrak{p})},\,R_{(\mathfrak{p})}\bigr)$.  Therefore, by recurrence,  
  $(E^\bullet_m)_\mathfrak{p} = (E^\bullet_1)_\mathfrak{p}$ for all $m\geq 1$.
    By Proposition~\ref{spectral1}, this  implies that
    $$\left(\gr_F  \SSExt^{\,^t}_{S\lQGr}(\MM,\,\XX) \right)_\mathfrak{p}= \left(E^\bullet_\infty\right)_\mathfrak{p}
     = \left(E^\bullet_{1}\right)_\mathfrak{p} = \SExt^t_{X} \left( \gr_F(\MM),\, \mathcal{O}_{X}\right)_\mathfrak{p},$$
     as required.
     \end{proof}

   \subsection{}\label{3.3.4}  The main result of this section is the following analogue of 
   \cite[Corollary~3.3.4]{Le}.
   
  \begin{corollary} Assume that $X$ is Gorenstein.
    Let $\MM\in S\lQGr$ with a good filtration $F$
    and pick $t\in \mathbb{N}$. Then
    $$\Char \SSExt^{\,^t}_{S\lQGr }(\MM,\,\XX)      \ = \ \bigcup_{i=1}^p  
    \VV(\mathfrak{p}_i) \cup \mathcal{D},$$
    where the $  \VV(\mathfrak{p}_i)$ are the irreducible components of 
 $\Char\MM$ of codimension $t$ in $X$ and $\mathcal{D}$ consists of a
    union of irreducible components of  codimension $\geq t$ in $X$, each of 
    which is contained in the union of the irreducible components of 
     $\Char \MM$   of codimension different from $t$.   
     \end{corollary}

    \begin{proof}   
    By  Corollary~\ref{subquotient}, and under the induced
     filtrations, we know that 
    $\gr \SSExt^{\,^t}_{S\lQGr}(\MM,\,\XX)$ is a subquotient of 
    $\SExt^ t_{X}(\gr_F  \MM,\,\mathcal{O}_X)$.
    Therefore, $$\Char \SSExt^{\,^t}_{S\lQGr}(\MM,\,\XX) \subseteq
     \Supp \SExt^t_{X} (\gr_F  \MM,\,\mathcal{O}_X).$$
  Lemma~\ref{3.3.3}(3) then implies that 
  $\Char \SSExt^{\,^t}_{S\lQGr }(\MM,\,\XX)     \subseteq  \bigcup_{i=1}^p  
    \VV(\mathfrak{p}_i) \cup \mathcal{D}$.

It  remains  to show 
    that each  irreducible component $\VV(\mathfrak{q})$ of $\Char\MM$ 
    of codimension $t$
    actually appears in $\Char \SSExt^{\,^t}_{S\lQGr}(\MM,\,\XX)$. 
    However, by Lemma~\ref{3.3.3}(3), such a $\VV(\mathfrak{q})$ is also an irreducible 
    component of   $\Supp \SExt^t_{X}(\gr_F \MM,\, \mathcal{O}_X)$.
 Therefore, by Proposition~\ref{3.3.4-extra}, 
 $\left(\gr \SSExt^{\,^t}_{S\lQGr}(\MM,\, \XX) \right)_{\mathfrak{q}}\not=0$,
 and so   $\VV(\mathfrak{q})$ does indeed appear in  
     $\Char \SSExt^{\,^t}_{S\lQGr}(\MM,\,\XX)$. 
    \end{proof}


\clearpage
\section{The double Ext spectral sequence.}\label{section-spectral}
\bigskip

 \begin{boxcode}    
 Fix  lower $\Lambdaplus$-directed algebras
 $S$ and $T$   and a finitely generated commutative $\Lambdaplus$-graded algebra
  $R = \bigoplus_{\lambda\in\Lambdaplus} R_\lambda$ that satisfy (H1--H3)  from \ref{filthyp}.
  Set $X = \Proj (R)$.
\end{boxcode}
\bigskip

  \subsection{} One of the most useful tools for applying homological techniques to enveloping algebras and other filtered algebras is the notion of an Auslander-Gorenstein ring and the related double Ext spectral sequence
   \cite[Chapter~2,Theorem~4.15]{Bj1}.
 As we prove in  Theorem~\ref{Bj4.15-NC},  these concepts have direct analogues for directed algebras.

    \subsection{Hypothesis}  \label{ample}   The following hypothesis  will be in 
    force for the rest of the section.

  \begin{enumerate}   
 \item[(H4)]  $X$ is Gorenstein, $\Spec R_0$ is normal, and the canonical morphism  $X  \to \Spec R_0 $ is birational. 

  \end{enumerate}
    
    \subsection{\bf Lemma}\label{EndRv}
\emph{Pick    $\nu\in \Lambda$ such that $R_\nu\not=0$. Then $R_0=\End_{R_0}(R_\nu ) $.}

\begin{proof}   By (H2), $R$ is a domain and so  the field of rational functions $\kk(X)$ equals 
$R[\mathcal{C}_R^{-1}]_0$, where $\mathcal{C}_R$ denotes the set of nonzero homogeneous elements of $R$. 
By (H4), $\kk(X)=Q(R_0)$, the field of fractions of $R_0$. Therefore
 $R_\nu x^{-1}\subseteq \kk(X)$ for any  $0\not=x\in R_\nu$ 
 and hence $R_\nu x^{-1}y\subseteq R_0$ for some $0\not=y\in R_0$.

Set $E=\End_{R_0}(R_\nu )$.   By the last paragraph, we may identify 
$E=\{\theta\in \kk(X) : \theta R_\nu\subseteq R_\nu\}\supseteq R_0.$
 For any $\omega\geq \nu$, (H2) implies that   $R_\omega = R_{\nu}R_{\omega-\nu}$ and hence 
 that $E R_\omega\subseteq R_\omega$. 
Therefore, if $I=R_{\geq \nu}$   then $E\subset F=\End_R(I)$.  Since $R$ is a $\Lambda$-graded domain, so is $F$. 
Moreover, $F$ is a finitely generated $R$-module 
since  $F\cong Fx\subseteq R$ for any 
$0\not=x\in R_\nu$. By the same observation, $R$ and $F$ have the same graded quotient ring
$R[\mathcal{C}_R^{-1}] = F[\mathcal{C}_F^{-1}]$, and hence $F[\mathcal{C}_F^{-1}]_0=\kk(X)=Q(R_0)$.

By restriction to degree zero, it follows that  
 $F_0$and hence $E$  are   finitely generated as  $R_0$-modules. Moreover, 
 $$Q(R_0)\subseteq Q(E)\subseteq Q(F_0)\subseteq F[\mathcal{C}_F^{-1}]_0 =Q(R_0).$$
As $R_0$ is integrally closed by (H4), this implies that $R_0=E$.
\end{proof}

\subsection{}
   The next result is fundamental to our approach, since it implies that 
   $\SSExt^{\bullet}(-,\XX)$ for  directed algebras plays the r\^ole of 
   $\Ext^{\bullet}(-,R)$ for a unital algebra $R$ and it is this that allows us to mimic 
   the homological approach of Gabber \cite{Le}. 
 Recall the definition of acyclic objects in $S\lqgr$ from Definition~\ref{acyclic-sect}.
 
 \begin{theorem}\label{S-vee2} Let $\nu\geq \lambda_S$. Then

  {\rm (1)} $\mathscr{S}_{\nu}^\vee$ is an acyclic sheaf in $T\lQGr$, and
  
  {\rm (2)}  $\mathscr{S}_{\nu}^{\vee\vee}\cong \mathscr{S}_{\nu}$ as objects in $S\lQGr$.
\end{theorem} 

\begin{proof}   (1)   Fix $j>0$. By
   Corollary~\ref{easier2},  $\gr_F \mathscr{S}^\vee_{\nu}    = \mathcal{O}_X(\nu)$  
     is a vector bundle on $X$ and so certainly  $\SExt^ j_{X}(\gr_F \mathscr{S}^\vee_{\nu}     ,\,\mathcal{O}_X)=0 $.
 Thus,   by Corollary~\ref{subquotient}(2),   $\gr_F \SSExt^{\,^j}_{T\lQGr}(\mathscr{S}_{\nu}^{\vee}, \mathcal{X}) = 0$, 
 where $F$ is the good  filtration defined by Proposition~\ref{spectral1}.
 Therefore  $ \SSExt^{\,^j}_{T\lQGr}(\mathscr{S}_{\nu}^{\vee}, \mathcal{X}) = 0$
  by  Lemma~\ref{gr-comment}.

(2) By Lemma~\ref{S-vee}(2),  
 $$(\mathscr{S}_{\nu}^{\vee})^{\ve} \cong \bigoplus_{\gamma\geq \lambda_S} \Hom_{T_{\phi,\phi}}
 ({S}_{\nu,0}\otimes_{S_{0,0}}{T}_{\phi,0},\, {S}_{\gamma,0}\otimes_{S_{0,0}}{T}_{\phi,0}) $$  
 for any  $\phi\geq \lambda_T = \lambda_S$. Thus for $\gamma\geq\nu  $,
  left multiplication by $S_{\gamma,\nu} $  induces a homomorphism
 \begin{equation} \label{doubledual}
  (\mathscr{S}_{\nu})_{\gamma} = S_{\gamma,\nu} \longrightarrow \Hom_{T_{\phi,\phi}}
 ({S}_{\nu,0}\otimes_{S_{0,0}}{T}_{\phi,0},\,  {S}_{\gamma,0}\otimes_{S_{0,0}}{T}_{\phi,0}) =
  (\mathscr{S}_{\nu}^{\vee})_{\gamma}^{\ve}.  \end{equation}
  This produces a homomorphism 
   $\Theta : \mathscr{S}_{\nu} \longrightarrow  \pi_S((\mathscr{S}_\nu^\vee)^\ve)= \mathscr{S}_{\nu}^{\vee \vee}$ in $S\lQGr$. 
In order to prove that $\Theta$ is an isomorphism, it suffices to prove the same for \eqref{doubledual}.

Set  $Z={S}_{\nu,0}\otimes_{S_{0,0}}{T}_{\phi,0} $.
Then Definition~\ref{good-defn}(3) implies that  \eqref{doubledual} is the multiplication  map 
 \begin{equation} \label{doubledual1}
 S_{\gamma,\nu} \longrightarrow \Hom_{T_{\phi,\phi}}
 (Z, {S}_{\gamma,\nu}\otimes_{S_{\nu,\nu}}Z).  \end{equation}
As $S_{\gamma,\nu}$ is a projective right $S_{\nu,\nu}$-module, we can write
 $S_{\gamma,\nu}\oplus P\cong S_{\nu,\nu}^{(m)}$  and so,
 in order  to prove that  \eqref{doubledual1} is bijective 
it suffices to prove that the multiplication map  $ S_{\nu,\nu}^{(m)}\to \Hom_{T_{\phi,\phi}}
 (Z, {S}_{\nu,\nu}^{(m)}\otimes_{S_{\nu,\nu}}Z)  $ is bijective. Equivalently, it  suffices to 
 prove that the action map 
$
  \theta:   S_{\nu,\nu}\to \Hom_{T_{\phi,\phi}}
 (Z, Z) $
     is bijective. Since multiplication by  $S_{\nu,\nu}$ preserves the tensor product filtration on 
     $Z=S_{\nu,0}\otimes_{S_{0,0}} {T}_{\phi, 0}$,  the  map $\theta$  is filtered and hence induces an 
     associated      graded morphism
      \begin{equation} \label{action2} \chi:
      \gr S_{\nu,\nu}  \ \buildrel{\gr\theta}\over\longrightarrow  \ \gr \Hom_{T_{\phi,\phi}}(Z,Z)\\
       \    \hookrightarrow  \
      \Hom_{\gr T_{\phi,\phi}}(\gr (Z), 
      \gr(Z)),
   \end{equation} 
      where the final inclusion follows from the observations in  Subsection~\ref{goodhom}.
   Here $\chi$ 
       is again the natural multiplication  map. 
       
       Finally, Hypothesis (H3) implies that 
$\gr (Z) = R_{\gamma+\phi}$ and so, by 
Lemma~\ref{EndRv},  $\chi$ is an  isomorphism for $\gamma\gg 0$.
 It follows that $\theta$  and hence  \eqref{doubledual}  are also 
 isomorphisms, as required.
 \end{proof} 
 
\subsection{}
We are now ready to state and prove the second spectral sequence; in essence it  shows that  the Auslander-Gorenstein condition and   related double Ext spectral sequence,  
that are so useful for unitary algebras, have a natural analogue for directed algebras.
We remark  that, as in  \cite[Section~I.1]{Bj2}, we  prove the result for rings of finite injective dimension. 
However, it will often be more convenient to mimic  the argument of \cite[Chapter~2, Theorem~4.15]{Bj1} which 
unnecessarily assumes finite global dimension.  

For this result, define 
$$\EE^{p,q} (\MM) \ = \  \SSExt^{\,^p}_{T\lQGr}\big(\SSExt^{\,q}_{S\lQGr}(\MM,\,\XX),\,\XX\big), \qquad
\text{for all $p,q\geq 0$.}$$

 \label{Bj4.15-NC} 
\begin{theorem}  Let $\MM\in S\lqgr $
and set $d=\dim X.$  Then
\begin{enumerate}
\item  For all $v<j$,  and all subobjects $\mathcal{N}\subseteq \SSExt^{\,j}_{S\lQGr}(\MM,\XX)$ one has  
$ \SSExt^{\,^v}_{T\lQGr}(\mathcal{N},\XX)=0$. 

\item If $\mathcal{P}\in U\lQGr$, for $U=S$ or $T$,  then $\SSExt^{\,j}_{U\lQGr}(\mathcal{P},\XX)=0$ for $j>\dim X$.
\item There is a spectral sequence  $${}^{I}\hskip-2.5pt{E}^{p,q}_2 \ = \ 
\EE^{p,q} (\MM)   \ \  \Longrightarrow \ \   
 \mathbb{H}^{p-q}(\MM),$$ 
where $ \mathbb{H}^{p-q}(\MM) = \MM$ if $p=q$ and $ \mathbb{H}^{p-q}(\MM) = 0$ otherwise.
\item There exists a filtration
$0=\RR_{-1}(\MM)\subseteq  \RR_0(\MM)\subseteq 
 \RR_1(\MM)\subseteq \cdots \subseteq \RR_d(\MM) = \MM$
whose sections $\MM_v = \RR_v(\MM)/\RR_{v-1}(\MM)$ appear in exact sequences
$$0\to \MM_v\to  \EE^{d-v,d-v} (\MM) 
\to \mathcal{W}_v\to 0.$$
The cokernel $\mathcal{W}_v$ is isomorphic to a subfactor of the direct sum of double Ext groups
of the form $\EE^{a,b} (\MM)$ for $(a,b)=(d-v+j+2,\, d-v+j+1)$ and $0\leq j\leq v-2.$ 
\end{enumerate}
\end{theorem}

\begin{proof}   We need two preliminary observations. Let $\mathcal{P}\in S\lqgr .$
 First,   by Corollary~\ref{3.3.4}, and in the notation of Definition~\ref{char-defn},
\begin{equation}\label{Bj7.3}
\chdim\big(\SSExt^{\,^j}_{S\lQGr}(\mathcal{P}, \XX)\big) 
\leq d-j \qquad\text{for all}\ 0\leq j\leq d.
\end{equation}
Next, Lemma~\ref{localise}  and \cite[Proposition~2.9]{Ba}  implies that, for any $\mathfrak{p}\in \Proj(R)$,
$$ \SExt^v_{X}(\gr_F\mathcal{P}, \mathcal{O}_X)_\mathfrak{p}  \ = \ 
 \SExt^v_{\mathcal{O}_{X,\mathfrak{p}}} ((\gr_F\mathcal{P})_{\mathfrak{p}}, \mathcal{O}_{X,\mathfrak{p}}) 
\ = \ 0 \qquad \text{if} \ v<d-\chdim(\mathcal{P}).$$ 
This implies that  
 $\SExt^v_X(\gr_F\mathcal{P}, \mathcal{O}_X) = 0$ for $v< d-\chdim(\mathcal{P})$.
  Thus, by Corollary~\ref{subquotient} and Lemma~\ref{gr-comment},  
\begin{equation}\label{Bj7.4}
\SSExt^{\,^v}_{S\lQGr}(\mathcal{P},\,\XX) = 0 \qquad\text{if}\   v < d-\chdim(\mathcal{P}).
\end{equation}
By symmetry, these two equations also hold for $\mathcal{P}\in T\lqgr.$  
 We now turn to the proof of the theorem.
 
   (1)    Applying   \eqref{Bj7.3} to  $\mathcal{P}=\MM$
   shows that $\alpha=\chdim \NN \leq \chdim  \SSExt^{\,^j}_{S\lQGr}(\MM,\,\XX)\leq  d-j$. Thus if 
   $\SSExt^{\,^v}_{S\lQGr}(\NN,\,\XX) \not=0$ then the analogue of \eqref{Bj7.4} for $\NN\in T\lqgr$
   shows that $v\geq d-\alpha\geq j$, as required.

  (2) Since $\injdim R=d$  by (H4), this is immediate from Corollary~\ref{subquotient}.

 (3,4)      By  Lemma~\ref{resolutions} and  the comments in Subsection~\ref{sheaf-filt-defn} 
 we may form a finitely generated projective resolution
  $\cdots \to \mathcal{P}_d\to\cdots\to \mathcal{P}_0 \to \MM\to 0$ in $S\lqgr$. Now  form the complex
  $$\mathbf{P}^\vee_{\bullet} : \qquad 0\to \mathcal{P}_0^\vee\to \mathcal{P}_1^\vee\to\cdots\to \mathcal{P}_{d}^\vee\to \cdots,\qquad
  \text{where} \quad\mathcal{P}_r^\vee=\SSHom_{S\lQGr}(\mathcal{P}_r,\,\XX).$$ By construction, there 
  exists some fixed $\nu\gg0$ such that   the $\mathcal{P}_r$ are direct sums of shifts of 
   $\mathscr{S}_{\nu}$ and so, by Theorem~\ref{S-vee2}(1),
  each  $\mathcal{P}_r$ is an acyclic sheaf in $T\lQGr$. Moreover, by   Lemma~\ref{S-vee}(1), 
  each  $\mathcal{P}_r^\vee\in T\lqgr$.
  We  claim that we can form a double complex
{\small{  \begin{equation}\label{complex}
  \begin{CD}   
   \vdots&&&\vdots&&&&&&&\vdots  \\
  @VVV & @VVV &  @.    && @VVV \\
  \mathcal{Q}_{10} &@>>>  \mathcal{Q}_{11} &@>>> & \cdots & @>>> \mathcal{Q}_{1d}  @>>>\cdots \\
    @VVV & @VVV & @. && @VVV \\
    \mathcal{Q}_{00} &@>>>  \mathcal{Q}_{01} &@>>> & \cdots & @>>> \mathcal{Q}_{0d}@>>>\cdots \\
    @VVV & @VVV & @. && @VVV \\
    \mathcal{P}^\vee_{0} &@>>>  \mathcal{P}^\vee_{1} &@>>> & \cdots & @>>> \mathcal{P}^\vee_{d}@>>>\cdots \\
    @VVV & @VVV & @. && @VVV  \\
    0&&&0&&&&&&&0
    \end{CD}
    \end{equation} }}
    in $T\lqgr$  satisfying the following conditions:
    \begin{itemize}
\item    Each $\mathcal{Q}_{ij}$ is a projective object in $T\lqgr$ and each column is exact;
    \item the cohomology groups  $\{\mathcal{H}_{vj}\}$ of the  row  complexes 
    $0\to \mathcal{Q}_{v0}\to \mathcal{Q}_{v1}\to\cdots\to \mathcal{Q}_{vd}\to \cdots$ are 
    projective;
    \item for each $j$,  the complex 
    $\cdots  \mathcal{H}_{2j}\to \mathcal{H}_{1j}\to \mathcal{H}_{0j}\to \SSExt^{\,^j}_{S\lQGr}(\MM,\,\XX)$
    is a projective resolution in $T\lqgr$ of $\SSExt^{\,^j}_{S\lQGr}(\MM,\,\XX)$.
    \end{itemize}
     The proof of this assertion is almost immediate. As is proved in \cite[pp.58-9]{Bj1},
starting with any complex of  finitely generated modules  over a noetherian ring in place of $\mathbf{P}_{\bullet}$,     then 
 such a double complex exists. 
     So it exists in  $T_{\gamma, \gamma}\lmod$ for $\gamma\gg 0$. But 
       $T_{\gamma,\gamma}\lmod \simeq T\lqgr$    by Proposition~\ref{corner-ring} 
       applied to $T$.  So use this equivalence to translate  the given complex $\mathbf{P}_{\bullet}$ 
      in  $T_{\gamma,\gamma}\lmod$  to a complex $\mathbf{C}_{\bullet}$ in $T  \lqgr$ and apply the above construction.

    Now  consider the double complex    
     $\bigl\{\mathcal{Q}^\vee_{jv} =\SSHom_{T\lQGr}(\mathcal{Q}_{jv},\,\XX)\bigr\}$
     of modules in $S\lQGr$.  We remark that the rest of the proof closely follows that of Bjork (\cite[pp.60--62]{Bj1}
     or \cite[pp.62--64]{Bj2}). We will cite both books since, although the former assumes finite global dimension,
      the arguments given there are closer to the ones we need.

 \begin{sublemma}\label{Bj4.15-sublemma}
  The hypercohomology groups in the double complex   
  $\bigl\{\mathcal{Q}^\vee_{jv} \bigr\}$ vanish everywhere except 
on the $d^{\mathrm{th}}$ diagonal, where the cohomology group is 
$\mathcal{H}^d \cong \MM$ as  objects in $S\lQGr$. 
\end{sublemma}
        
 \begin{proof} Consider 
 \eqref{complex}. By   Theorem~\ref{S-vee2} and  construction,  the $\mathcal{P}_r^\vee$ and 
  $\mathcal{Q}_{uv}$ are acyclic sheaves, while   the columns of \eqref{complex}  are exact.  Thus   the dual complex 
$0\to \mathcal{P}_{j}^{\vee\vee} \to \mathcal{Q}_{0j}^\vee\to \mathcal{Q}_{1j}^\vee\to \cdots$ will be exact in 
$S\lQGr$.  Moreover, by Theorem~\ref{S-vee2}(2), $\mathcal{P}_j^{\vee\vee}\cong \mathcal{P}_j$. 
 Up to a change of notation, the proof of  
 \cite[Proposition~1.2]{Bj2} may now be used without further change to 
 prove the sublemma.
 \end{proof}
 
 The remainder of proof of \cite[Chapter~2, Theorem~4.15]{Bj1}, from the   sublemma on page 60
 through to its conclusion on page 61,   may now be used to prove the
 rest of  parts~(2) and (3) of the present theorem;
  the only changes being to substitute  projective module by acyclic sheaf,
 $(-)^*$ by $(-)^\vee$, and hence $\Ext(-, A)$ by $\SSExt(-,\XX)$. 
The proof only uses the fact that $A$ has finite injective dimension rather than finite global dimension,  and the relevant analogue of that assertion is given by part (2) of this theorem.
     \end{proof}
    
 \subsection{}    By combining  Theorem~\ref{Bj4.15-NC}(3)  with \eqref{Bj7.3}
    we obtain
    
\begin{corollary}   \label{Bj7.4bis} 
   For $\MM\in S\lqgr$ and $0\leq v\leq d$, define   $\MM_v$,  $ \RR_v(\MM) $ and $\mathcal{W}_v$ 
    as in   Theorem~\ref{Bj4.15-NC}.  Then
    \begin{enumerate}
    \item  $\chdim(\MM_v)\leq v$.
    \item  $\chdim\bigl( \mathcal{W}_v\bigr)\leq d-(d-v+2)=v-2$.
    \item  $\chdim\bigl(\RR_v(\MM)\bigr)\leq v$ for all $0\leq v\leq d.$
   \qed
 \end{enumerate}
 \end{corollary}

 
 \clearpage
   \section{Equidimensionality.}\label{section-main} 
   \bigskip

   \begin{boxcode} Fix  lower $\Lambdaplus$-directed algebras $S$ and $T$ with $T_{00}^{\opp} = S_{00}$, 
 and a finitely generated commutative $\Lambdaplus$-graded algebra
  $R = \bigoplus_{\lambda\in\Lambdaplus} R_\lambda$.
 We assume that these satisfy hypotheses (H1--H4) from \ref{filthyp} and \ref{ample}. 
 Write $X = \Proj (R)$, and set $\dim X=d$.
 \end{boxcode}
 
\bigskip

  \subsection{} \label{s-homog-defn}  We are now ready to prove the main result of the paper:   
   the  characteristic variety $\Char\MM$ of a simple object $\MM\in S\lqgr$ is equidimensional.
The proof  also works for  the following more general modules.
 
    An object $\MM\in S\lqgr$ will be   called {\it $s$-homogeneous}\index{$s$-homogeneous objects in $S\lqgr$}
   if each nonzero subobject $\NN\subseteq \MM$ satisfies $\chdim(\NN)=s$, in the sense of Definition~\ref{char-defn}.  
   It is not clear how to characterise 
   $s$-homogeneous objects, but at least if $\MM$ is  simple then 
   $\MM$ is   $s$-homogeneous for some~$s$. 
  Here is another example, the proof of which is left to the reader:   if $P$ is a prime ideal of $S_{\lambda,\lambda}$ for a good parameter $\lambda$,   then the image 
   $\Psi(S_{\lambda,\lambda}/P)\in S\lqgr$ is  $s$-homogeneous  for some~$s$.

\subsection{Definition}   \label{j-defn}
If  $\MM\in S\lqgr$,
define the    \emph{grade} $j(\MM)$ of $\MM$ 
   to be \ $$j(\MM)=\min\{j : \SSExt^{\,^j}_{S\lQGr}(\MM,\, \XX)\not=0\}.$$ Recall the notation $\EE^{p,q} (\MM)=
     \SSExt^{\,^p}_{T\lQGr}\big(\SSExt^{\,^q}_{S\lQGr} (\MM,\,\XX),\,\XX\big) $
   from \ref{Bj4.15-NC}.

\begin{proposition}\label{Bj7.1} Let  $\MM\in S\lqgr$.  
\begin{enumerate}
\item 
$j(\MM)+\chdim(\MM)= d.$
   
   \item  Let $p\in \mathbb{N}$. Then
   $\EE^{p,p}(\MM)$ is either zero or $(d-p)$-homogeneous.
   
\item
 If $\chdim(\MM)\leq d-v$ for some  $v\in \mathbb N$, then $\chdim\bigl(\EE^{j,v}(\MM)\bigr)\leq d-j-2$ for all $j>v$.
   \end{enumerate}
   \end{proposition}
   
   \begin{proof}  (1) By \eqref{Bj7.4} we know that $j(\MM)\geq d-\chdim(\MM)$, 
   so we only need to prove the opposite inequality.
  Suppose that $v>d-j(\MM)$.  Then $d-v<j(\MM)$ and  $\SSExt^{\,^{d-v}}_{S\lQGr}(\MM,\,\XX)=0$.
  Therefore, $0=\EE^{d-v,d-v}(\MM)=\MM_v$. This  implies that 
    $\MM=\RR_{d-j(\MM)}(\MM)$   and so  Corollary~\ref{Bj7.4bis}  gives 
  $ \chdim(\MM)\leq d-j(\MM)$.

   (2)  This is a formal consequence of   the spectral sequence  Theorem~\ref{Bj4.15-NC}(2). 
   Modulo replacing $\Ext^i( -,\,A)$ 
   by $\SSExt^{\,^i}(-,\XX)$, the proof of \cite[Proposition~1.18 and Corollary~1.20]{Bj2},
    or indeed of  \cite[Chapter~2, Theorem~7.10]{Bj1}, 
   can be used mutatis mutandis.
   
   (3)   This result, which will not be needed in this paper,   is again proved by investigating the spectral sequence from 
   Theorem~\ref{Bj4.15-NC}(2). Modulo replacing $\Ext^i( -,\,A)$ 
   by $\SSExt^{\,^i}(-,\XX)$, the proof of \cite[Chapter~2, Lemma~7.11]{Bj1} 
   can be used verbatim.  
   \end{proof}

   \subsection{} At last, we reach our destination.
  The proof of this result follows   that of \cite[Th\'eor\`eme~ 3.3.2]{Le}.

   \begin{theorem}\label{main-thm} Suppose that 
    $\MM\in S\lqgr$ is $s$-homogeneous for some $s\in \mathbb{N}.$  Then each irreducible 
   component of  $\Char \MM$ has dimension $s$.
\end{theorem}

\begin{proof}
Keep the notation from Theorem~\ref{Bj4.15-NC} and let 
  $p\in\mathbb{Z}$. If    $\EE^{p,p}(\MM)  =0$, 
then   $\MM_{d-p}=0$ by Theorem~\ref{Bj4.15-NC}(3). Conversely, if $\MM_{d-p}=0$ then 
$\EE^{p,p}(\MM)=W_{d-p}$ and so, by Corollary~\ref{Bj7.4bis}(2),
$$\chdim\bigl(\EE^{p,p}(\MM)\bigr) =\chdim\bigl( W_{d-p}\bigr)\leq d-p-2.$$  By 
Proposition~\ref{Bj7.1}(2), this implies that $\EE^{p,p}(\MM)=0$.

Suppose that there exists an irreducible component $\VV(\mathfrak{p})$ of $\Char \MM $ with 
$\dim \VV(\mathfrak{p})< s$.  Let $t=\text{codim}\bigl(\VV(\mathfrak{p})\bigr)$ and consider  
 $\EE=\SSExt^{\,^t}_{S\lQGr }(\MM,\,\XX)\in T\lqgr$.  
By  Corollary~\ref{3.3.4}, $\VV(\mathfrak{p})$ is  an
 irreducible component of $\Char \EE$.
 Applying  Corollary~\ref{3.3.4} again, but with $\MM$ replaced by  $\EE$,  
shows that $\VV(\mathfrak{p})$ is also  an irreducible component of 
$\SSExt^{\,^t}_{T\lQGr}(\EE,\,\XX) =\EE^{t,t}(\MM)$. 
In particular, $\EE^{t,t}(\MM)\not=0$.  
  By the first paragraph of this proof, it follows  that   $\MM_{d-t}\not=0$
  and hence that $\RR_{d-t}(\MM)\not=0$.
 By Corollary~\ref{Bj7.4bis},  $\chdim \RR_{d-t}(\MM )\leq d-t <s$. 
 This contradicts the fact that $\MM$ is $s$-homogeneous. Hence each component of $\Char \MM$ has dimension at least $s$.

Conversely since $\chdim (\MM) = s$ each component of $\Char \MM$ has dimension at most $s$. This completes the theorem.
\end{proof}


\section{Applications to quantizations of symplectic singularities}\label{symplectic-sect}

 \subsection{} \label{BPW-defns}
  We end by showing that   many important examples of $\Z$-algebras arising from geometric representation theory   satisfy Hypotheses (H1--H4). As an application, we answer a question from \cite{GS2} on rational Cherednik algebras.
  
\subsection{} We start by recalling results presented in \cite{BPW}.
Let $Y$ be an affine irreducible symplectic singularity with a $\C^*$-action and  $\pi: X\longrightarrow Y$ a 
$\C^*$-equivariant symplectic resolution of singularities. We assume first that $\C^*$ acts on $\C[Y]$ with non-negative weights and such that $\C[Y]^{\C^*} = \C$, and second that there exists $0<m \in \mathbb{N}$ such that the $\C^*$-action scales the given 
symplectic form $\omega$ by $t\cdot \omega = t^{m} \omega$ for all $t\in \C^*$.  
This is called a \emph{conical symplectic resolution.} As we saw in the introduction, there are many interesting examples of these varieties.

\subsection{}
Given $\lambda\in H^2(X,\C)$ there is a unique sheaf of $\C^*$-equivariant complete $\C[[h]]$-algebras $\mathcal{Q}_X^\lambda$ on $X$ with the following properties:  
\begin{enumerate}
\item $\mathcal{Q}_X^\lambda/h\mathcal{Q}_X^\lambda \cong \mathcal{O}_X$;
\item the Poisson bracket on $X$ induced by $\omega$ equals the Poisson  bracket defined 
by $\{ f_1, f_2 \} \equiv h^{-1} [ \hat{f}_1 , \hat{f}_2] \text{ (mod $h$)}$, where $\hat{f}_1, \hat{f}_2$ are lifts to $\mathcal{Q}_X^{\lambda}$ of $f_1, f_2 \in \mathcal{O}_X$;   
\item $t\cdot h = t^{-m} h$.
\end{enumerate}

Let $\mathcal{D}_X^{\lambda} = \mathcal{Q}_X^{\lambda}[h^{-1/m}]$, a sheaf of $\mathbb{C}((h^{1/m}))$-algebras on 
$X$. Write $\mathcal{D}_X^{\lambda}\md$ for  the full subcategory of $\C^*$-equivariant 
$\mathcal{D}_X^{\lambda}$-modules whose objects $\mathcal{M}$ admit a 
$\C^*$-equivariant $\mathcal{Q}_X^{\lambda}$-lattice $\mathcal{M}(0)$ such that 
$\mathcal{M}(0)/h\mathcal{M}(0)$ is a coherent $\mathcal{O}_X$-module, \cite[\S~4]{BPW}. 
Setting $U_{\lambda} = (\mathcal{D}_X^{\lambda}(X))^{\C^*}$, there is a functor 
$\Gamma^{\C^*} : \mathcal{D}_X^{\lambda}\md \longrightarrow U_{\lambda}\md$
which takes $\mathcal{M}$ to the $\C^*$-invariant global sections $\Gamma (X,\mathcal{M} )^{\C^*}$. This has a left adjoint 
${\sf Loc}: U_{\lambda}\md \longrightarrow \mathcal{D}_X^{\lambda}\md$ which sends $N$ to $\mathcal{D}_X^{\lambda}\otimes_{U_{\lambda}} N$. The pair $(X,\lambda)$ 
is called \emph{localizing} if $(\Gamma^{\C^*}, {\sf Loc})$ are mutually inverse equivalences.

\subsection{} \label{sec:BPWalgebra} There is an $\mathbb{N}$-algebra attached to $(X, \lambda)$ and the 
choice of a   very ample line bundle $\mathcal{L}$ on $X$. To construct it, let $\eta\in H^2(X,\C)$ be 
the first Chern class of $\mathcal{L}$. Then, for 
 any pair of integers $k,m$, \cite[Proposition~5.2]{BPW} shows that there is a unique sheaf of 
 $\mathbb{C}^*$-equivariant $(\mathcal{Q}_X^{\lambda + k\eta}, \mathcal{Q}_X^{\lambda+ m\eta})$-bimodules 
 ${}_{k\eta}\mathcal{B}_{m\eta}(\lambda)$ which quantizes the line bundle $\mathcal{L}^{k-m}$. Set
  $${}_{k\eta}S_{m\eta}(\lambda) = \Gamma^{\C^*}( {}_{k\eta}\mathcal{B}_{m\eta}(\lambda)[h^{-1/m}] )
   \in (U_{\lambda+k\eta}, U_{\lambda+m\eta})\text{-bimod}. $$
   Then    the desired  $\mathbb{N}$-algebra is
   $$S(\lambda, \eta) = \bigoplus_{k\geq m \geq 0} {}_{k\eta}S_{m\eta}(\lambda),$$   
 with multiplication induced from tensor products (see \cite[Definition~5.5]{BPW}).
 
\subsection{}
For each $p\in \mathbb{N}$,  the following new    $\mathbb{N}$-algebras can be constructed  from $S$: 
$$ S^{(p)}(\lambda, \eta) =   \bigoplus_{k\geq m \geq 0} {}_{pk\eta}S_{pm\eta}(\lambda)
\qquad\text{and}\qquad
S_{\geq r}(\lambda, \eta)  =  S(\lambda,\eta)_{\geq r} =  \bigoplus_{k\geq m \geq r} {}_{k\eta}S_{m\eta}(\lambda).$$
The former of these algebras is called the $p^{\text{th}}$ \emph{Veronese ring} of $S(\lambda, \eta)$.

\begin{proposition}
\label{propBPW}
{\rm (\cite[Propositions~5.10 and 5.14]{BPW})} Let $\lambda \in H^2(X,\C)$ and let $\eta\in H^2(X,\C)$ be the first 
Chern class of a relatively very ample line bundle on $X$. 
\begin{enumerate}
\item There exists $q\gg 0$ such that  $S_{\geq r}(\lambda, \eta) $ is Morita for all $r\geq q$.
\item $(X,\lambda)$ is a good parameter if and only if $S^{(p)}(\lambda, \eta)$ is Morita for $p\gg 0$. 
\end{enumerate}
\end{proposition}

\subsection{} The algebras constructed  by Proposition~\ref{propBPW} satisfy 
the Hypotheses (H1--H4) from Subsections~\ref{filthyp}
and~\ref{ample}.

\begin{proposition}\label{smpl-hold}
 Let $\lambda \in H^2(X,\C)$ and let $\eta\in H^2(X,\C)$ be the first Chern class of a relatively very ample line bundle on $X$. Assume that $(X,\lambda)$ is localizing. Then Hypotheses (H1--H4) hold
 for a  pair of directed algebras $(S^{(p)}(\lambda, \eta),T)$, where the  integer $p$ and algebra $T$ 
 are  constructed within the proof.
\end{proposition}
\begin{proof}
By Proposition~\ref{propBPW} we can find $p\gg 0$ such that $S^{(p)}(\lambda, \eta)$ is Morita. Now consider 
$(\mathcal{Q}_X^{\lambda})^{op}$, another $\C^*$-equivariant quantization of $\mathcal{O}_X$. By 
\cite[Proposition~3.2]{BPW} it is isomorphic to $\mathcal{Q}_X^{-\lambda}$. Therefore, 
$U_{\lambda}^{op}\cong U_{-\lambda}$, and so repeating the above construction, {\it with the same line bundle 
$\mathcal{L}$}, we produce another $Z$-algebra $S(-\lambda, \eta)$. It may be that $(X,-\lambda)$ is not localizing. 
By Proposition~\ref{propBPW}, however, we do know that for $q\gg 0$ the $\mathbb{N}$-algebra 
$$S_{\geq q}(-\lambda,\eta) = \bigoplus_{k\geq m \geq q}  {}_{k\eta}S_{m\eta}(-\lambda)$$ is Morita. 

So we will consider the following algebras, where $r$ is the least common multiple of $p$ and $q$.
$$S = \bigoplus_{a\geq b \geq 0} {}_{ra\eta}S_{rb\eta}(\lambda, \eta), \quad T =
 \bigoplus_{a\geq b \geq 0} {}_{ra\eta}S_{rb\eta}(-\lambda, \eta).$$

(H1) By construction, $S$ is Morita and so all values are good for it, while $T_{\geq1}$ is Morita.

(H2) As observed in \cite[Section~5.2]{BPW},  $\gr S =\gr T = \tR$ where 
$R_a = H^0(X, \mathcal{L}^{\otimes ra})$. Now it may be that   $R_m R_n \not= R_{m+n}$. But 
  $X$ is smooth, hence normal, and so this property will  hold if we replace $R$ by some further  Veronese subring. 
Also, since $X$ is reduced and irreducible, certainly $R$ is a domain. 
So, after passing to the appropriate Veronese ring, (H2) holds.

(H3) We must prove that $\gr (S_{m0}\otimes_{S_{00}} T_{n0}) \cong R_{m+n}$ for all $m,n$.   
By Lemma~\ref{EndRv}, each  $R_m$ is a torsion-free rank one $R_0$-module. Since  
    $\gr {}_{rm\eta}S_{0}(\lambda)=  \gr {}_{rm\eta}S_{0}(-\lambda) = R_m$,  it follows that both 
    ${}_{rm\eta}S_{0}(\lambda)$ and ${}_{rn\eta}S_{0}(-\lambda)$ are torsion-free of rank one on both sides. 
 Since $S$ is Morita,  certainly $S_{m0} = {}_{rm\eta}S_{0}(\lambda)$ is a projective ${}_0S_0(\lambda)$-module, so
   the tensor product 
      ${}_{rm\eta}S_{0}(\lambda) \otimes_{{}_0S_0(\lambda)} {}_{rn\eta}S_{0}(-\lambda)$ is therefore  torsion-free of rank one.

    As observed earlier, $\mathcal{Q}_X^{-\lambda} \cong (  \mathcal{Q}_X^{\lambda})^{op}$. Thus
     the $(\mathcal{Q}_X^{-\lambda+rn\eta}, \mathcal{Q}_X^{-\lambda})$-bimodule
     $  {}_{rn\eta}\mathcal{B}_{0}(-\lambda)$ may be considered as a $(\mathcal{Q}_X^{\lambda}, \mathcal{Q}_X^{\lambda-rn\eta})$-bimodule quantizing $\mathcal{L}^{\otimes rn}$, and as such must be isomorphic to
    ${}_{0} \mathcal{B}_{-rn\eta}(\lambda)$  by the uniqueness of quantization property mentioned in \ref{sec:BPWalgebra}.
     Hence, by
   \cite[Proposition~5.2]{BPW},  multiplication produces an isomorphism of 
   $(\mathcal{Q}_X^{\lambda+rn\eta}, \mathcal{Q}_X^{\lambda - rm\eta})$-bimodules 
   $${}_{rm\eta}\mathcal{B}_{0}(\lambda)\otimes_{\mathcal{Q}_X^{\lambda}} {}_{rn\eta}
   \mathcal{B}_{0}(-\lambda)
   \ = \  {}_{rm\eta}\mathcal{B}_{0}(\lambda)\otimes_{\mathcal{Q}_X^{\lambda}}{}_{0}
   \mathcal{B}_{-rn\eta}(\lambda)
 \   \stackrel{\sim}{\longrightarrow}  \ {}_{rm\eta}   \mathcal{B}_{-rn\eta}(\lambda).$$
    Taking invariant global sections then produces a non-zero multiplication map 
    $$\mu:  S_{m0}\otimes_{S_{00}}T_{n0} \longrightarrow {}_{rm\eta}S_{-rn\eta}(\lambda).$$  
       By torsion-freeness, $\mu$ must   be injective. But $\mu$  is also filtered surjective because  
     $$gr S_{m0} \cdot \gr T_{n0} = R_m \cdot R_n = R_{m+n} = \gr {}_{rm\eta}S_{-rn\eta}(\lambda).$$ 
       Thus $\gr(S_{m0}\otimes_{S_{00}} T_{n0}) = R_{m+n}$, as required.

(H4) Symplectic singularities are automatically normal and so (H4)  follows from the fact that  $X\longrightarrow Y$
is a resolution of singularities.
\end{proof}

\subsection{} \label{char-defn2}
Given a $U_{\lambda}$-module $N$, set $\mathcal{N} = {\sf Loc}(N)$. We define $\Char_{\sf loc}(N)$ to be 
the support on $X$ of the coherent sheaf $\overline{\mathcal{N}} :=\mathcal{N}(0)/h^{1/m}\mathcal{N}(0)$ where 
$\mathcal{N}(0)$ is a lattice for $\mathcal{N}$. This is well-defined, independent of the choice of lattice. \index{$\Char$, 
the characteristic variety of!a deformation-quantization sheaf}

We claim that  $\Char_{\sf loc}(N) = \Char(\Psi(N))$, where $\Char(\Psi(N))$ is defined using the $\Z$-algebra in
 Proposition~\ref{smpl-hold}, as in Definition~\ref{char-defn}. Indeed the discussion preceding Theorem 5.6 in \cite{BPW} together with the first two 
 paragraphs of the proof of \cite[Proposition~5.9]{BPW}  show that there is an isomorphism 
 $$\bigoplus_{a\geq 0} \Gamma^{\C^*}({}_{ra\eta}\mathcal{B}_0(\lambda)[h^{-1/m}]\otimes_{\mathcal{D}_X^{\lambda}}
  \mathcal{N})\stackrel{\sim}{\longrightarrow} \bigoplus_{a\geq 0} {}_{ra\eta}S_0(\lambda) \otimes_{U_{\lambda}} N 
  = \Psi(N)$$ and furthermore that a lattice $\mathcal{N}(0)$ for $\mathcal{N}$ induces a good filtration on 
  $\Psi(N)$ such that 
  $$\gr (\Psi(N)) = \bigoplus_{a\geq 0} \Gamma(X, \overline{\mathcal{N}} \otimes \mathcal{L}^{\otimes ra})$$
   This shows two things: first that, by its goodness, this filtration on $\Psi(N)$ may be used to calculate 
   $\Char (\Psi(N))$; second that this equals the support of $\overline{\mathcal{N}}$. In other words $\Char_{\sf loc}(N)$.

It follows that $\Char(N)$ is independent of the choice of $\eta$ and of Veronese algebra in Proposition~\ref{smpl-hold}.      
\label{mainexampleequiresult}  
So combining this with Theorem~\ref{main-thm} gives:

\begin{corollary} Let $\lambda\in H^2(X,\C)$ and suppose that $(X,\lambda)$ is localizing. If $M$ is either a simple 
$U_{\lambda}$-module or a prime factor of $U_\lambda$, then  the characteristic variety $\Char (M)$  is equidimensional.
\end{corollary}

\subsection{Rational Cherednik algebras}\label{97}
 We apply the above analysis to the symplectic resolution 
$$\pi: X = \hi \longrightarrow Y = {\sf Sym}^n \C^2,$$ where $\pi$ is the Hilbert-Chow morphism. We take the 
$\C^*$-action which is induced from dilation on $\C^2$. For $\lambda \in H^2(\hi,\C) = \C$ there is an explicit 
realization of the algebra $U_{\lambda}= (\mathcal{D}_X^{\lambda}(X))^{\C^*}$ as the spherical rational Cherednik 
algebra of type $A$, defined as follows. 

Let $W={S}_n$ be  the
symmetric group and write    $e = \frac{1}{|W|}\sum_{w\in W} w\in \C W$ for  the 
trivial idempotent.  Let $\h = \C^n$ denote  the
permutation representation of $W$ and 
   $$\hr = \{ (z_1, \ldots , z_n)\in\h : z_i\neq z_j \text{ for all }1\leq i < j\leq n\}$$
   the subset of $\h$ on which $W$ acts freely.   
 
Define  ${H}_{\lambda}$ to be the subalgebra of  $\D(\hreg)\rtimes W$ generated by $W$, 
the vector space $\h^*=\sum x_i\C\subset \C [\h]$ of linear functions,  and the {Dunkl operators}
\begin{equation} \label{dunkdef}
D_\lambda(y_i) \ = \  \frac{\partial}{\partial x_i} \ -\  \frac{1}{2}\sum_{j \neq k} 
(\lambda - \frac{1}{2})\frac{\langle y_i, x_j-x_k\rangle}{x_j - x_k} (1 - s_{jk})\end{equation} 
where $\{y_1, \ldots , y_n\} \subset \h$ is  the dual basis to $\{x_1, \ldots , x_n\}$
and  the $s_{jk}\in W$ are simple trans\-positions. The  \emph{spherical subalgebra} is  defined to be 
$e{H_{\lambda}}e$: thanks to \cite[Theorem~2.8]{GGS} it is isomorphic to $U_{\lambda}$.
\footnote{In much of the 
literature the rational Cherednik algebra with parameter $\lambda$ is what we call $H_{\lambda+\frac{1}{2}}$.} 
One can also  define $H_\lambda$ in terms of the reflection representation of  
$S_n$ instead of $\C^n$ However,  as noted in \cite{GS1}, the two theories are exactly parallel and the same results hold in the two cases.

Set $\mathcal{C} =  \{z + \frac{1}{2}: z=\frac{m}{d}\ \mathrm{where}\ m, d \in \Z\text{ with }  2\leq d\leq n, \text{ and } z\notin \Z
\}.$ By
 \cite[Theorem~1.6]{GS1} and Proposition~\ref{propBPW}, $(X,\lambda)$ is localizing provided 
 $\lambda\notin \mathcal{C}\cap (-1/2,1/2)$.
 (We remark that \cite{GS1} has the additional restriction that $\lambda\not\in\mathbb{Z}$, but this is removed using  
   \cite[Theorem~4.3]{BE} and \cite[Remark~3.14(1)]{GS1}.)
       
\subsection{}\label{98} Combining this discussion with  Corollary~\ref{mainexampleequiresult} gives the following result.

\begin{corollary}\label{equi-cor} Assume that $\lambda \notin \mathcal{C}\cap (-1/2,1/2)$.
  Let $M$ be a simple   $U_{\lambda}$-module
 or a prime factor ring of $U_{\lambda}$. 
Then the characteristic variety $\Char M\subseteq \hi$ is equidimensional.\qed
\end{corollary}
 
This answers   \cite[Question~4.9]{GS2} and \cite[Problem~8]{Ro}; by \cite[Theorem~7.6]{GGS},
 the definition of a characteristic variety given there  agrees with the definition given here.

\subsection{} We end with an amusing consequence of the above corollary  for which we need to assume 
that $\lambda\not\in \mathbb{Z}$. As in \cite[(2.7.1)]{GS2}, write  ${\bf ch}(M)$ for the characteristic cycle of a module 
$M\in \mathcal{O}_\lambda(S_n)$, the category $\mathcal{O}$ for the rational Cherednik algebra. 
By \cite[Theorem~6.7]{GS2},  the cycles are fully supported   on the subvarieties 
 $Z_{\mu} \subset \hi$ defined in \cite[Section~6.4]{GS2} where $\mu$ is a partition of $n$. Write ${\bf ch}_{\mu}(M)$ for the 
 multiplicity of ${\bf ch}(M)$ along $Z_{\mu}$. 

As in \cite[6.10 and 6.11]{GS2}, there is  an isomorphism
$\chi: K_0(\mathcal{O}_\lambda(S_n)) \longrightarrow \Lambda_n$ where $\Lambda_n$ is the degree n part of the 
ring of symmetric functions. It is defined by
$$ \chi(M) = \sum_{\mu} {\bf ch}_{\mu} (M) m_{\mu}$$ where $m_{\mu}$ is the monomial 
symmetric function associated to $\mu$. 
Under this isomorphism it was shown in [loc.cit] that $\chi (\Delta_\lambda(\mu)) = s_{\mu}$, the Schur function 
associated to
 $\mu$. The same is true for the co-standard modules $\nabla_\lambda(\mu)$ since they have the same image in 
 $ K_0(\mathcal{O}_\lambda(S_n))$  as $\Delta_\lambda(\mu)$.

Since  $\lambda\notin \Z$ we may apply  \cite[Theorem~6.13]{Rq2}
which shows that $S_q(n,n)\md$ and $\mathcal{O}_\lambda(S_n)$ are 
equivalent. Here $S_q(n,n)$ is the $q$-Schur algebra at $q=-\exp(2\pi \sqrt{-1} \lambda)$, which is the equivalent to the category
 of representations of the quantum group $G_q(n)$ of degree $n$. By taking characters we have a mapping
  $$\chi_q : S_q(n,n)\md \longrightarrow \Lambda_n.$$
The category $S_q(n,n)\md$ has Weyl modules, written $\nabla_q(\mu)$: the Weyl character formula shows that   $\chi_q(\nabla_q(\mu)) = s_{\mu}$. 
  
Putting these observations together gives an equivalence  
$$\Theta : S_q(n,n)\md \longrightarrow \mathcal{O}_\lambda(S_n)$$ 
which sends Weyl modules to co-standard modules, so intertwines $\chi_q$ with $\chi$. Now, in $\chi_q(V)$, the 
coefficient $m_{\mu}$ is {\it by definition} the multiplicity of the weight space $V_{\mu}$. 
We therefore deduce:

\begin{proposition}\label{known-unknown} Assume $\lambda \notin \mathcal{C}\cap (-1/2,1/2)$ and that $\lambda\notin \Z$. Then, in the above notation,
${\bf ch}_{\mu}(M) = \dim \Theta^{-1}(M)_{\mu}$ for any partition $\mu$ of $n$.
\end{proposition}


 \section*{Index of Notation}\label{index}
\begin{multicols}{2}
{\small  \baselineskip 14pt

Acyclic sheaf \hfill\eqref{acyclic-sect}

$\Char(M) $ \hfill\eqref{char-defn}, \eqref{char-defn2}

$\chdim(M)$ \hfill\eqref{char-defn}

Directed algebras, \hfill\eqref{2.1}

Duals $\MM^\ve$, $\MM^\vee$  of $\MM\in S\lQGr$ \hfill\eqref{homfil}

$\tR$; a directed algebra\hfill\eqref{Ex2.2}

Ext groups $\SSExt$ for $S\lQGr$    \hfill\eqref{ext-defn}, \eqref{agreement}

Filtrations and good filtrations  \hfill\eqref{sheaf-filt-defn}

Filtered projective resolutions \hfill\eqref{resolution-chat}  

Good parameter \hfill\eqref{good-defn}

Hypotheses (H1--H3)\hfill\eqref{filthyp}

Hypothesis (H4)\hfill\eqref{ample}

$j(M)$; the grade of $M$ \hfill\eqref{j-defn}

Locally left noetherian  \hfill\eqref{projgen1}

$M_{\ast,\gamma}$  \hfill\eqref{ext-defn}

Morita directed algebras  \hfill\eqref{good-defn}

$S\lGr$,   $S\lgr$; categories of   modules\hfill\eqref{graded-defn}

$S\lQGr$, $S\lqgr$; quotient categories \hfill\eqref{lackofsymmetry}

$s$-homogeneous modules \hfill\eqref{s-homog-defn}

${\sf S}_{\ast,\lambda}$; distinguished objects in $S\lQGr$  \hfill\eqref{hommumjum}

 $\mathscr{S}_{\lambda}=
 \pi_S({\sf S}_{\ast, \lambda}) $  \hfill\eqref{projective}

Shift $S_{\geq \lambda}$   \hfill\eqref{shiftlemma}

Support of a module \hfill\eqref{support-defn}

Tensor product filtration\hfill\eqref{filt-notation}

$\XX$, $\XX_\gamma$; dualising sheaves \hfill\eqref{acyclic-sect}}
\end{multicols}


\end{document}